\documentclass[11pt]{amsart}
\usepackage{amsmath,amsfonts,amssymb,amsthm}
\usepackage{epsfig}
\usepackage{graphics}
\usepackage{subfig} 
\usepackage{comment}
\usepackage{hyperref}
\usepackage[margin=1in]{geometry}

\newtheoremstyle{pak}{9pt}{9pt}{\itshape}{}{\bfseries}{}{.5em}{}

\theoremstyle{pak}
\newtheorem{thm}{Theorem}[section]
\newtheorem{cor}[thm]{Corollary}
\newtheorem{lem}[thm]{Lemma}
\newtheorem{sublem}[thm]{Sublemma}

\newtheoremstyle{defin}
  {9pt}{9pt}{}{}{\bfseries}{}{.5em}{}
\theoremstyle{defin}

\newtheorem{cl}{Claim}
\newtheoremstyle{exm}
  {9pt}{9pt}{}{}{\scshape}{}{.5em}{}
\theoremstyle{exm}

\newtheoremstyle{proof}
  {}{}{}{}{\itshape}{:}{.5em}{}
\theoremstyle{proof}

\def\cC{\mathcal C}

\def\co{\mathcal O}

\def\cR{\mathcal R}
\def\cT{\mathcal T}

\def\cT{\mathbf{T}}

\def\<{\langle}
\def\>{\rangle}

\def\width{\text{{\bf wd\ts }}}
\def\height{\text{{\bf ht\ts }}}

\def\0{{\mathbf 0}}

\def\implies{\cRightarrow}

\def\.{\hskip.06cm}
\def\ts{\hskip.03cm}

\def\T{{{\mathbf{T}}}}
\def\R{{{\mathbf{R}}}}
\def\W{{{\mathbf{W}}}}
\def\M{{{\mathfrak{M}}}}

\def\l{\ell}
\def\co-NP{\textup{co-NP}}
\def\NP{\textup{NP}}
\def\SP{\textup{\#P}}
\def\G{\Gamma}
\def\cC{\mathcal{C}}
\def\cT{\mathcal{T}}
\def\cR{\mathcal{R}}
\def\Z{\mathbb{Z}}
\def\N{\mathbb{N}}
\def\e{e}
\def\p{\pi}
\def\t{\tau}
\def\s{\sigma}

\newcommand{\abs}[1]{\left\lvert#1\right\rvert}
\newcommand{\type}{t}
\newcommand{\latin}[1]{\textsl{#1}}

\def\implies{\Rightarrow}

\newcommand{\problem}[1]{\textsc{#1}}
\newcommand{\sct}[1]{\problem{Simply Connected $#1$-Tileability}}

\newcommand{\problemdef}[3]{
\medskip
\begin{tabular}{ll}
\multicolumn{2}{l}{\problem{#1}}\\
\textbf{Instance:} & #2 \\
\textbf{Decide:} & #3
\end{tabular}\medskip}

\title[Tiling simply connected regions with rectangles]{Tiling simply connected regions with rectangles}

\author[Igor~Pak]{ \ Igor~Pak$^\star$}

\author[Jed~Yang]{ \ Jed~Yang$^\star$}

\thanks{\thinspace ${\hspace{-.45ex}}^\star$Department of Mathematics, UCLA, Los Angeles, CA 90095, USA; \.
\texttt{\{pak,jedyang\}@math.ucla.edu}}
\date{July 4, 1776}

\begin{document}
\date{}

\begin{abstract}
In~\cite{BNRR}, it was shown that tiling of general regions with
two rectangles is NP-complete, except for a few trivial special cases.
In a different direction, R\'{e}mila~\cite{Rem2} showed that for
simply connected regions by two rectangles, the tileability can be
solved in quadratic time (in the area).  We prove that there is a finite
set of at most $10^6$ rectangles for which the tileability problem of
simply connected regions is \NP-complete, closing the gap between positive
and negative results in the field.  We also prove that counting such
rectangular tilings is \SP-complete, a first result of this kind.
\end{abstract}
\maketitle

\section{Introduction}
The study of \emph{finite tilings} is a classical subject of
interest in both theoretical and recreational literature~\cite{Gol1,GS}.
In the \emph{tileability problem}, a finite set of tiles~$\T$~is fixed,
and a region is an input.  This problem is known to be polynomial
in some cases, and \NP-complete in others (see~\cite{pak-horizons}).
Over the years, the hardness results were successively simplified
(in statement, not in proof), with both sets of tiles and the
regions becoming more restrictive.  This paper is a new step
in this direction.

In~\cite{BNRR}, it was shown that tiling of general regions
with two bars is NP-complete, except for the case of dominoes.
In a different direction, R\'{e}mila~\cite{Rem2} (building on
the ideas in~\cite{KK,Thu}), showed that for \emph{simply connected
regions} and two rectangles, the tileability can be solved in
quadratic time (in the area).  The following theorem closes the gap
between these polynomial and \NP-complete results.

\begin{thm}[Main Theorem] \label{thm-rect}
There exists a finite set $\R$ of at most $10^6$ rectangular
tiles, such that the tileability problem of simply connected
regions with~$\R$ is \NP-complete.
\end{thm}

Our proof of the Main Theorem is split into two parts.  In the
first part, we use the language of \emph{Wang tiles} to reduce the
\problem{Cubic Monotone $1$-in-$3$ SAT} problem, known to be \NP-complete,
to the $\T$-tileability of simply connected regions with Wang tiles.
In the second part, we reduce Wang tileability to tileability with
rectangular tiles.  Both our reductions are \emph{parsimonious} and are
used to prove that counting the number of tilings of simply connected
regions is also hard, via reduction from~\problem{2SAT}.

\begin{thm} \label{thm-sharp}
There exists a finite set $\R$ of at most $10^6$ rectangular tiles,
such that counting the number of tilings of simply
connected regions with~$\R$ is \SP-complete.
\end{thm}

Although \SP-completeness is known for tilings of general
regions with right tromino and square tetromino~\cite{MR},
nothing was known for tilings with rectangles.
We refer to
Section~\ref{s:fin} for the history of the problem,
references, and further remarks.

\medskip

\section{Definitions and basic results}\label{s:def}

\subsection{Ordinary tiles}
Consider the integer lattice $\Z^2$ as a union of closed unit squares with pairwise disjoint interiors.
A \emph{region} is a finite union of such unit squares such that the interior is connected.
An (\emph{ordinary}) \emph{tile} is a finite simply connected region.


A \emph{tileset} $\T$ is a collection of tiles.
Given a region $\G$ and a tileset $\T$, a \emph{$\T$-tiling} of $\G$ is a union
of translated copies of tiles from $\T$ with pairwise disjoint interiors covering $\G$.
If a region admits a $\T$-tiling then it is \emph{$\T$-tileable}.
We may simply say \emph{tiling} and \emph{tileable} when $\T$ is understood.
Consider the following decision problems regarding tileability:

\problemdef{Simply Connected Tileability}
{Simply connected region $\G$, finite tileset $\T$.}
{Whether $\G$ is $\T$-tileable?}

\problemdef{\sct{\T}}
{Simply connected region $\G$.}
{Whether $\G$ is $\T$-tileable?}

An input region can be given by the (finite) union of the squares it contain.
The following is one of the early \NP-completeness results~\cite{GJ}.

\begin{thm}\label{thm-2-dim}
If both region $\G$ and tileset $\T$ are part of the input,
\problem{Simply Connected Tileability} is \NP-complete in the plane.
\end{thm}

For the rest of the paper,
we will focus on finding a fixed $\T$ such that \sct{\T} is \NP-complete.
The following result is an extension of Theorem~\ref{thm-2-dim}.
\begin{thm}\label{thm-ordinary}
There exists a set $\T$~of~$\ts 23$ tiles, such that \sct{\T} is \NP-complete.
\end{thm}

The proof follows an explicit construction of Wang tiles (see below).
While we do not use Theorem~\ref{thm-ordinary}, it is of independent interest, 
and the intermediate results in its proof provide a key step towards the proof 
of the Main Theorem.  The history behind this theorem and its potential 
generalizations is outlined in Subsection~\ref{ss:fin-levin}.

\subsection{Wang tiles}
The \emph{edges} of an ordinary tile are the unit-length edges on the boundary.
Given a set of colors and an ordinary tile $\t$, a \emph{generalized Wang tile} is an assignment of colors to the edges of $\t$.
Note that an (\emph{ordinary}) \emph{Wang tile} is a generalized Wang tile of a unit square.
The region $\G$ we are trying to tile will also have specified colors on its boundary.
A region is (\emph{Wang}) \emph{tileable} if there is a tiling where incident edges have the same color, including on the boundary of the region (see Figure~\ref{wang-tiling}).
If a tileset consists of (generalized) Wang tiles, tileability always mean Wang tileability.
\begin{figure}[hbtp]
   \includegraphics[scale=0.4]{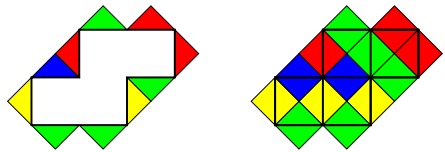}
   \caption[A colored region and a Wang tiling.]{A colored region (left) and a Wang tiling (right). Colored edges are drawn as triangles for visibility.}
   \label{wang-tiling}
\end{figure}

\subsection{Relational Wang tiles}
Let us consider a more general setting.  A set of \emph{relational Wang tiles}
is a collection $\W$ of squares and the following data.  The vertical
(respectively horizontal) \emph{Wang relation} $V_\W(\t,\t')$
(respectively $H_\W(\t,\t')$) specify that $\t'\in\W$ is allowed to be placed
immediately below (respectively to the right of) $\t\in\W$.
We suppress the subscripts when it can be understood from context.
The \emph{boundary tiles} of a region $\G$ is a map from the exterior
edges of $\G$ to the tiles~$\W$.
By abuse of language, we define the notion of tiling in this context:
\ts a \emph{$\W$-tiling} of a region $\G$ is a map $\p:\G\to W$ such
that tiles placed next to each other satisfy the Wang relations.
Whenever a tile is adjacent to an exterior edge, we check the
Wang relations as if the boundary tile corresponding to the edge
is on the other side of the edge.

\subsection{Complexity}
Throughout the paper we consider many tiling problems that are \NP-complete.
All these problems are trivially in~\NP.
Indeed, given a description of a tiling, one could simply check if it is in fact a tiling.
To prove \NP-hardness, we reduce a known \NP-complete problem to the problem in question.
We refer to~\cite{GJ,Pap} for definitions and details.

We will embed \problem{Cubic Monotone $1$-in-$3$ SAT} as a tiling problem.
Let $X=\{x_1,\ldots,x_n\}$ be a set of boolean variables.
A (\emph{monotone $1$-in-$3$}) \emph{clause} $C$ is a set of three variables.
A (\emph{cubic monotone $1$-in-$3$}) \emph{expression} $E$ is a finite collection
$\cC$ of monotone $1$-in-$3$ clauses, where each variable $x_i\in X$ occurs three times.
We say such $E$ is (\emph{$1$-in-$3$}) \emph{satisfiable} if there is an
assignment of boolean values $\{0,1\}$ to the variables $x_i\in X$
such that each clause in $E$ contains precisely one variable receiving~$1$
(and thus two variables receiving~$0$).

\problemdef{Cubic Monotone $1$-in-$3$ SAT}
{Set $X$ of variables, cubic monotone expression $E$.}
{Whether $E$ is $1$-in-$3$ satisfiable?}

The following result was shown by Gonzalez in the language of exact covers:

\begin{thm}[\cite{Gon}]\label{thm-mr}
\problem{Cubic Monotone $1$-in-$3$ SAT} is \NP-complete.\footnote{Given an expression $E$, we can associate a bipartite graph $G$ with vertex set $X\sqcup\cC$,
where a variable $x\in X$ is adjacent to a clause $C\in\cC$ if $x\in C$.
Moore and Robson showed something stronger in \cite{MR}, that this problem is \NP-complete even if we require the associated graph to be planar.
They did this by reducing from \problem{Planar $1$-in-$3$ SAT}, which is \NP-complete \cite{Lar,MuR}.
However, we do not need to use the planar version.}
\end{thm}
We will reduce \problem{Cubic Monotone $1$-in-$3$ SAT} to a tiling problem \sct{\T} for some fixed $\T$.

\subsection{Counting problems}
Throughout the paper we consider natural counting problems corresponding to the decision problems.
For example, instead of asking whether satisfying assignments exist,
we ask \emph{how many} satisfying assignments there are.
Similarly, for tileability, we count the number of tilings.
If in the proof of \NP-completeness,
the corresponding reductions give a bijection between the sets of solutions,
we call such reduction \emph{parsimonious}.

Parsimonious reductions have the additional benefit of proving counting results using the same reduction.
The class \SP{} consists of the counting problems associated with decisions problems in~\NP{}.
A counting problem is \SP-complete if it is in \SP{} and every \SP{} question can be reduced to it.
Thus, if there is a parsimonious reduction from problem~$Q_1$ to~$Q_2$,
then if $Q_1$ is \SP-complete, then so is~$Q_2$.
We refer to~\cite{Val} (see also~\cite{Pap}) for definitions and details on~\SP{} complexity class.

One main goal is to reduce \problem{Cubic Monotone $1$-in-$3$ SAT} to a tiling problem
\sct{\T} for some fixed $\T$.
This reduction will turn out to be parsimonious, hence
the number of satisfying assignments of a given instance of the satisfiability problem
can be calculated by counting the number of tilings of the transformed instance.

However, it is not known whether the associated counting problem
\problem{\#Cubic Monotone $1$-in-$3$ SAT} is \SP-complete.
To get the \SP-completeness result in Theorem~\ref{thm-sharp},
we will modify the reduction to use \problem{2SAT} instead,
whose associated counting problem \problem{\#2SAT} is \SP-complete.

\medskip

\section{Reduction lemmas} \label{sec-reduction}

\subsection{Basic reductions}
In this section we consider five classes of \problem{Tileability} problems.
Let $\cT$ be a collection of tiles and $\cR$ be a collection of regions.
A decision problem in \emph{$(\cT,\cR)$-\problem{Tileability}} consists
of a \emph{fixed} tileset $\T\subset\cT$,
receives some $\G\in\cR$ as input, and outputs whether $\G$ is $\T$-tileable.

We say $(\cT,\cR)$-\problem{Tileability} is \emph{linear time reducible} to $(\cT',\cR')$-\problem{Tileability}
if for any finite tileset $\T\subset\cT$, there exists a finite tileset $\T'\subset\cT'$
and a \emph{reduction map} $f:\cR\to\cR'$ that is computable in linear time (in the complexity of $\G\in\cR$),
such that $\G\in\cR$ is $\T$-tileable if and only if $f(\G)$ is $\T'$-tileable.%
\footnote{Recall that the tiles in the input are given as collections of unit squares.}
If, moreover, that $(\cT',\cR')$-\problem{Tileability} is linear time reducible to
$(\cT,\cR)$-\problem{Tileability}, then they are \emph{linear time equivalent}.
Note that the transformation of the tilesets need not be efficient nor bijective.

For instance, if $\cT$ is the collection of all rectangular tiles and $\cR$ consists of simply connected regions,
then $(\cT,\cR)$-\problem{Tileability} is a class of problems regarding tiling simply connected regions with rectangular tiles.
To simplify the notation, we drop the prefix in $(\cT,\cR)$-\problem{Tileability} when the sets~$\cT$ and~$\cR$ are understood.

\begin{lem}[Tileability Equivalence Lemma]\label{lem-reduction}
The following five classes of \problem{Simply Connected Tileability} problems are linear time equivalent:
\begin{enumerate}
\item \problem{Tileability} with a fixed set of rectangular tiles.
\item \problem{Tileability} with a fixed set of ordinary tiles.
\item \problem{Tileability} with a fixed set of generalized Wang tiles.
\item \problem{Tileability} with a fixed set of ordinary Wang tiles.
\item \problem{Tileability} with a fixed set of relational Wang tiles.
\end{enumerate}
Moreover, the size of the tileset can be preserved in the reductions between
\textup{(ii)} and \textup{(iii)}.
\end{lem}

\begin{proof}
The reductions (i)$\implies$(ii)$\Leftrightarrow$(iii)$\implies$(iv)$\implies$(v)
are elementary and given below. The reduction (v)$\implies$(i) is stated separately
as Lemma~\ref{lem-rect} and proved in the next section.

We may consider a rectangular tile as an ordinary tile,
which in turn is a monochromatic generalized Wang tile.
Therefore the reductions (i)$\implies$(ii)$\implies$(iii) are immediate,
where each reduction map is simply the identity.

(iii)$\implies$(iv).
Given a set of generalized Wang tiles, color each interior edge with a new color not used anywhere else,
and consider each square as a separate ordinary Wang tile (see Figure~\ref{iii-iv}).
These tiles are forced to reassemble themselves as the original generalized Wang tiles.
The reduction map is again the identity.

(iv)$\implies$(v).
It is obvious how to define the Wang relations to mimic the
colored Wang tiles without increasing the number of tiles.
To encode the boundary conditions, we may need to introduce
less than $4\chi$ tiles, where $\chi$ is the number of colors permitted
on the boundary.  Indeed, to specify a color $c$ on the top boundary,
we need to choose an (arbitrary) tile whose bottom color is~$c$.
If no such tile exists, we must add a new tile to do so.
If we do not involve the new tile in any Wang relations in the
other directions, then it will never be used in the actual tiling,
and thus will not affect tileability.
We do the same for the other three directions.

The final reduction (v)$\implies$(i) is more difficult and is the
content of Lemma~\ref{lem-rect} and proved in a later section.

To preserve the number of tiles in (iii)$\implies$(ii),
scale the generalized Wang tile and replace each colored edge by an
appropriate rectilinear zig-zag curve to encode the matching rules
(see Figure~\ref{iii-ii} and \cite{Gol2}).
\end{proof}

\begin{figure}[hbtp]
   \subfloat{\includegraphics[scale=0.4]{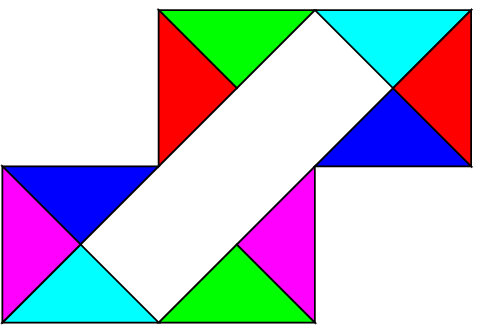}}
   \quad
   \subfloat{\includegraphics[scale=0.4]{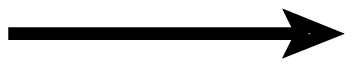}}
   \quad
   \subfloat{\includegraphics[scale=0.4]{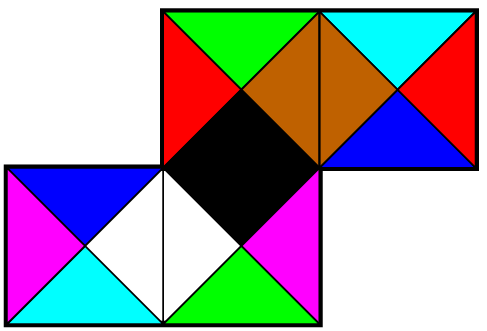}}
   \quad
   \subfloat{\includegraphics[scale=0.4]{arrow-small}}
   \quad
   \subfloat{\includegraphics[scale=0.4]{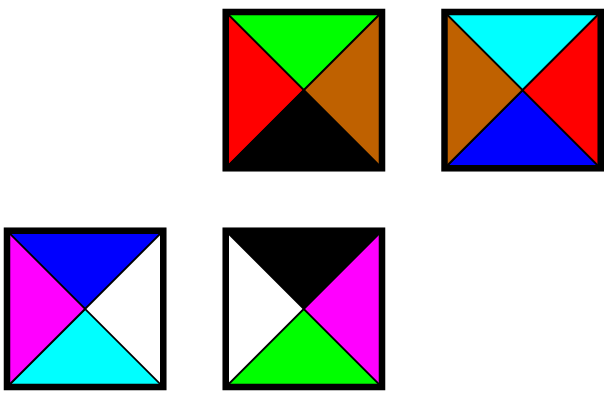}}
   \caption{From generalized Wang tiles to ordinary ones.}
   \label{iii-iv}
\end{figure}

\begin{figure}[hbtp]
   \subfloat{\includegraphics[scale=0.6]{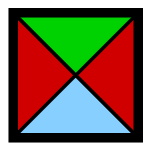}}
   \qquad
   \subfloat{\includegraphics[scale=0.4]{arrow-small}}
   \qquad
   \subfloat{\includegraphics[scale=0.6]{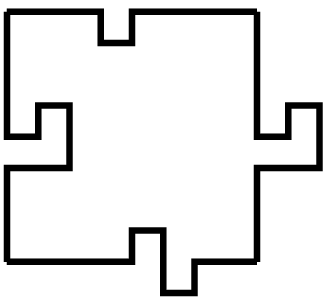}}
   \caption{Replacing each colored edge by a zig-zag curve to get ordinary tiles.}
   \label{iii-ii}
\end{figure}

\medskip

\subsection{Two main reductions}

\begin{lem}[First Reduction Lemma]\label{lem-wang}
There exists a set $\T$ of at most $23$ generalized Wang tiles with total area $133$ and using $9$ colors such that \sct{\T} is \NP-complete.
Moreover, this will be achieved by a parsimonious reduction from \problem{Cubic Monotone $1$-in-$3$ SAT}.
\end{lem}

\begin{lem}[Second Reduction Lemma]\label{lem-rect}
For a set $\W$ of at most $k$ (ordinary) Wang tiles with~$c$
(boundary) colors, there exists a set $\R$ of at most $8(k+4c)^2$
rectangular tiles with the following property.
Given a simply connected colored region $\G$, there is
a simply connected region $\G'$ such that $\G$ is $\W$-tileable
if and only if $\G'$ is $\R$-tileable.
Moreover, this reduction is parsimonious and can be computed in linear time.
\end{lem}

We may transform the set of $23$ generalized Wang tiles afforded by Lemma~\ref{lem-wang},
according to the procedure outlined in (iii)$\implies$(ii) of Lemma~\ref{lem-reduction},
in order to obtain Theorem~\ref{thm-ordinary} using $23$ ordinary tiles.
Similarly, using the transformation of Lemma~\ref{lem-rect}, we conclude
the result for rectangular tiles in Theorem~\ref{thm-rect} (see Subsection~\ref{ss:proof-thm-rect}).
Theorem~\ref{thm-sharp} can be shown by modifying the proof of Lemma~\ref{lem-wang}
to achieve a parsimonious reduction from, say, \problem{2SAT},
whose associated counting problem is \SP-complete (see Subsection~\ref{ss:proof-thm-sharp}).

\medskip

\section{Proof of the First Reduction Lemma {\small(Lemma~\ref{lem-wang})}} \label{s:first-reduction}

\subsection{General setup}
The goal of this section is to construct a set of generalized
Wang tiles that could be used to solve \problem{Cubic Monotone $1$-in-$3$ SAT}.
Each expression will be encoded as a colored rectangular boundary.
Tiles corresponding to variables and clauses will appear on the left
and right sides of the region, respectively.  The variable tiles will
``transmit'' its state ($0$ or~$1$) through ``wires'' to the clause tiles;
each clause tile will ``check'' if precisely one out of three signals it
receives is~$1$.  The path of the transmissions will be regulated by placing
``crossover tiles'' that allow signals to crossover at specific locations.
The positioning of such tiles will be enforced by using a combination of
``control tiles'' that follow instructions encoded on the boundary.
Empty spaces will be filled by ``filler tiles.''

\subsection{Tileset $\T$}
Let $\T$ be a tileset with the $7$ small tiles shown in Figure~\ref{tiles-0}
and the $3$ big tiles in Figure~\ref{tiles-1}.  Some horizontal edges are
colored by their labels; all unlabeled edges are colored by~$0$, which is omitted
in the figures for clarity, but acts as any other ordinary color.
\begin{figure}[hbtp]
   \centering
   \subfloat[$L$]{\label{tile-l}\includegraphics[scale=0.4]{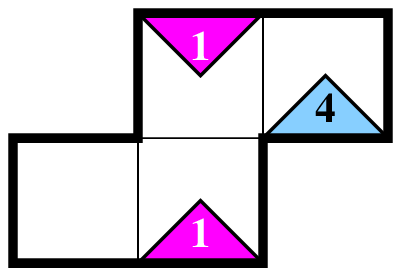}}
   \quad
   \subfloat[$W$]{\label{tile-w}\includegraphics[scale=0.4]{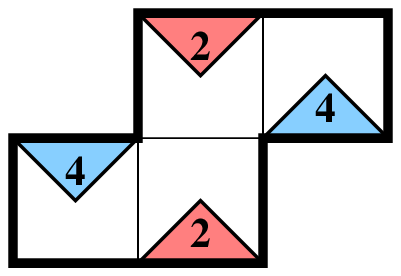}}
   \quad
   \subfloat[$S$]{\label{tile-s}\includegraphics[scale=0.4]{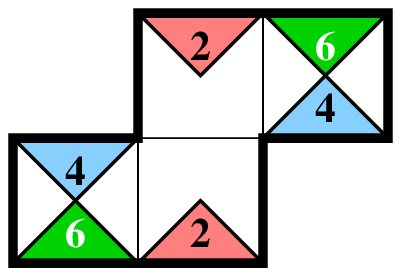}}
   \quad
   \subfloat[$R$]{\label{tile-r}\includegraphics[scale=0.4]{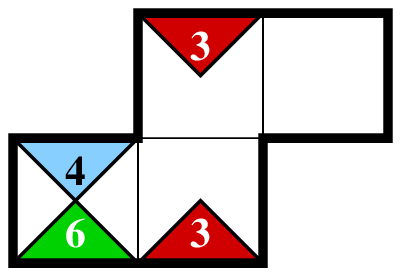}}
   \\\medskip
   \subfloat[$B$]{\label{tile-b}\includegraphics[scale=0.4]{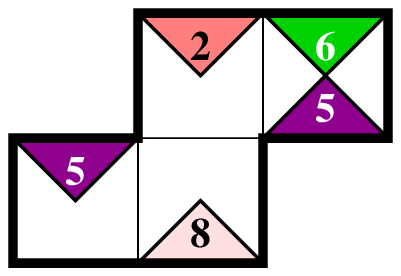}}
   \qquad
   \subfloat[$K$]{\label{tile-k}\includegraphics[scale=0.4]{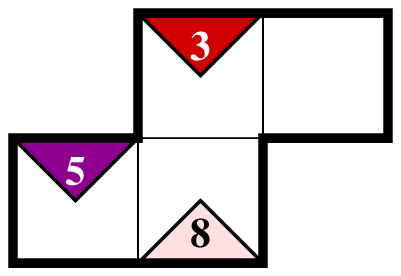}}
   \qquad
   \subfloat[$F$]{\label{tile-f}\includegraphics[scale=0.4]{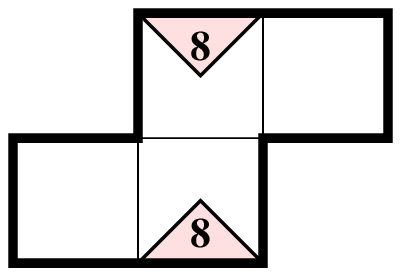}}
   \caption{Tiles in tileset $\T$.}
   \label{tiles-0}
\end{figure}

\begin{figure}[htbp]
   \centering
   \subfloat[$X$]{\label{tile-x}\includegraphics[scale=0.4]{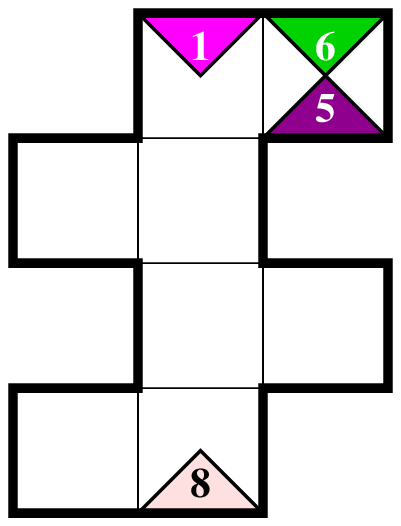}}
   \qquad
   \subfloat[$V$]{\label{tile-v}\includegraphics[scale=0.4]{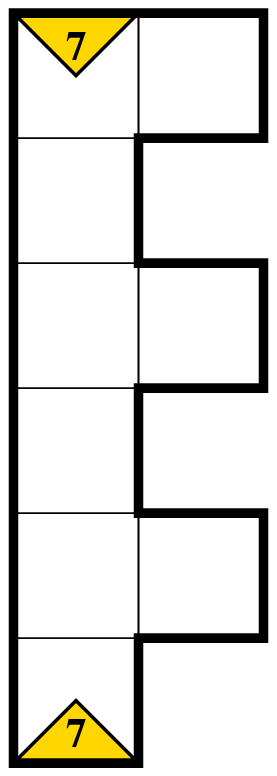}}
   \qquad
   \subfloat[$C$]{\label{tile-c}\includegraphics[scale=0.4]{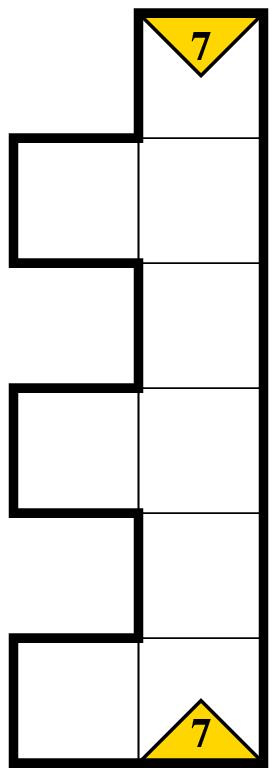}}
   \caption{More tiles in $\T$.}
   \label{tiles-1}
\end{figure}

\subsection{Tileset $\T'$}
Recall that the vertical edges of our tiles in $\T$ are all colored with~$0$.
Form $\T'$ by recoloring the vertical edges of tiles in $\T$ as follows.
Given each small tile $\t\in\T$ in Figure~\ref{tiles-0}, we introduce
a variant by coloring all its vertical edges with~$1$.
The color of the vertical edges is called the \emph{parity} of~$\t$.
Include both this variant and the original in~$\T'$.

Given a rectangular array of these tiles, the parities are consistent
across each row and are independent across the columns.  Intuitively,
these tiles act as wires that can transmit data (parity of the tile)
horizontally across the region.

We continue defining $\T'$.
We add three new versions of the crossover tile~$X$ as in Figure~\ref{tile-xd}.
Intuitively, this allows the data transmissions to \emph{crossover}.
We also add a variant of the variable tile~$V$, as in Figure~\ref{tile-vd},
where all the right vertical edges are colored with~$1$.
The parity of the variable tile corresponds to the truth value assigned to that variable.
Finally, we \emph{replace} the clause tile~$C$ by the three shown in Figure~\ref{tile-cd},
where each tile has one out of three pairs of left vertical edges colored with~$1$.
Thus $\T'$ consists of $23$ tiles.
\begin{figure}[hbtp]
   \centering
   \subfloat[$X$]{\label{tile-xd}\includegraphics[scale=0.4]{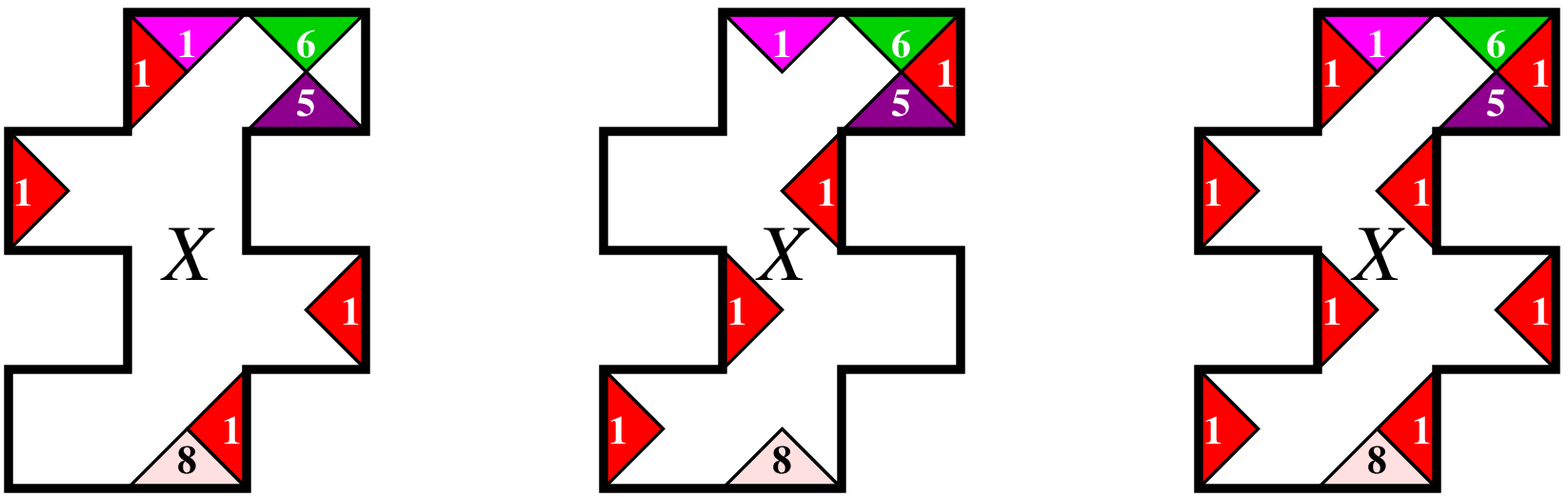}}
   \\
   \subfloat[$V$]{\label{tile-vd}\includegraphics[scale=0.4]{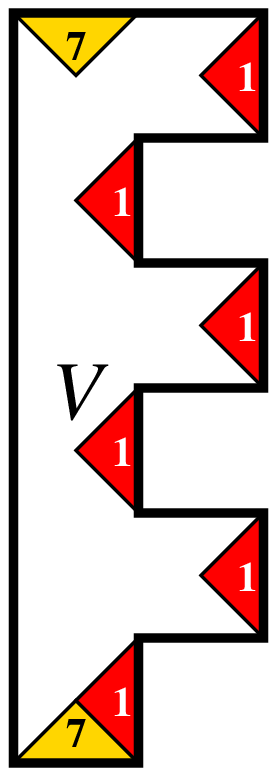}}
   \qquad\qquad
   \subfloat[$C$]{\label{tile-cd}\includegraphics[scale=0.4]{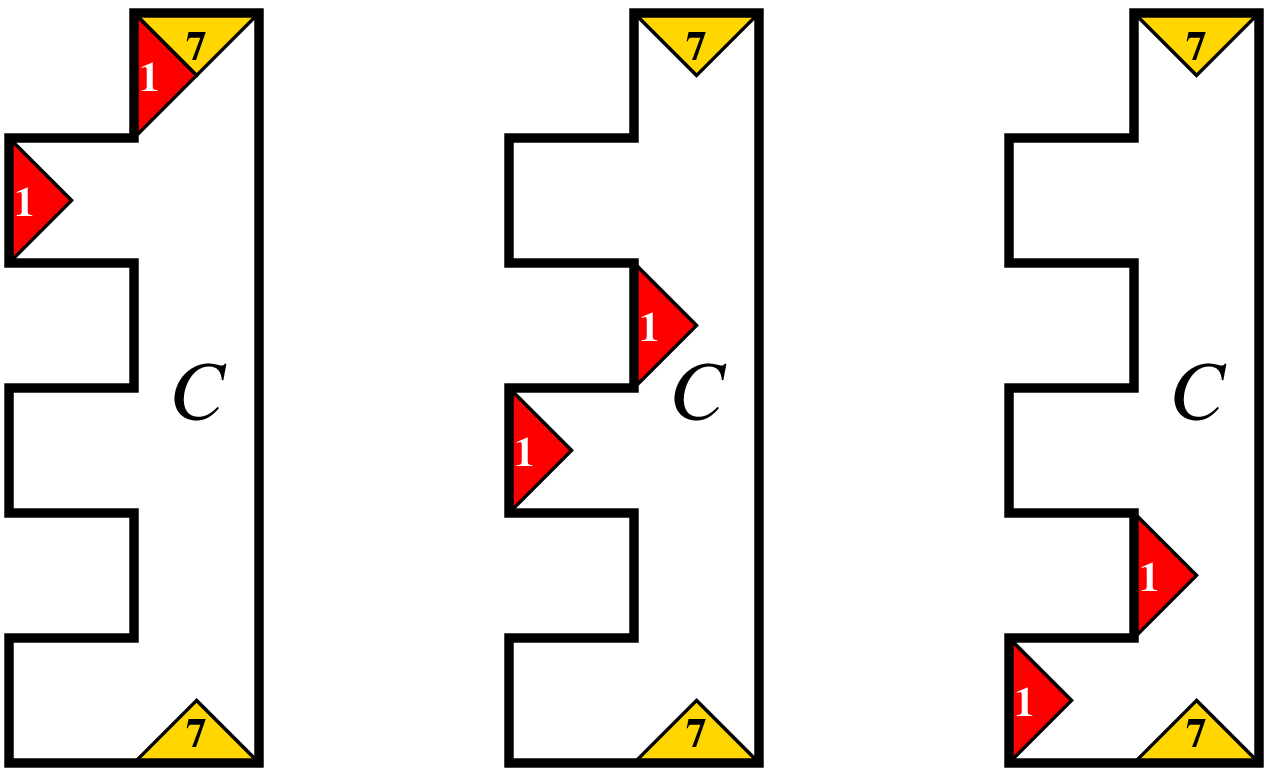}}
   \caption{Variations of tiles in $\T$.}
   \label{tile-2f}
\end{figure}

We will place the variable tiles on the left and the clause tiles on the right.
It remains to send the data from the variables to the correct clauses.
We achieve this by specifying boundary colors to force crossover tiles
to appear at the desired locations.

\subsection{Reduction construction}
Our goal is to embed the decision problem \problem{Cubic Monotone $1$-in-$3$ SAT}
as a tiling problem.  Given a cubic monotone $1$-in-$3$ SAT expression~$E$
with variables $X=\{x_1,\ldots,x_n\}$ and clauses $\cC=\{C_1,\ldots,C_n\}$,
consider it as a permutation $\s=\s_E\in S_{3n}$ in the symmetric group on
$3\ts n$ letters as follows.  Think of $\s$ as a bijection from the ordered
multiset $X'=\{x_1,x_1,x_1,x_2,\ldots,x_n\}$ to the ordered multiset
$\cC'=\{C_1,C_1,C_1,C_2,\ldots,C_n\}$, where each variable and each
clause is listed three times.  For each $x_i\in C_j$, we have $\s(x_i)=C_j$ once.
Now identify each multiset with $[3n]=\{1,2,\ldots,3n\}$ to get $\s$ as a permutation
in~$S_{3n}$.  Let $s_i=(i,i+1)$ be an adjacent transposition for $i\in[3n-1]$.
Write $\s=s_{i_1}s_{i_2}\ldots s_{i_d}$ as a product of adjacent transpositions,
with $d=O(n^2)$.\footnote{For illustration purposes, it is often convenient
to encode the product of adjacent transpositions using \emph{wiring diagrams},
as shown in Figures~\ref{wiring-1} and~\ref{wiring-2}.}

Let $c_k$ be the color sequence $01(02)^{k-1}63$.
Define a rectangular region $\G=\G_E$ as follows.
The height of $\G$ is $6n$ and the vertical edges are colored with~$0$.
The width is the length of the color sequence $7c_{i_1}c_{i_2}\ldots c_{i_d}07$,
which is used as the top boundary.  The bottom boundary is $7(08)^t07$ with the same length as the top boundary.
The following result demonstrates the ability to place the crossover tile~$X$ at
arbitrary depth of a large rectangular region.

\begin{sublem}\label{sl:unique-wang}
The region $\G$ admits a unique $\T$-tiling.
\end{sublem}


\begin{proof}
The left and right sides are forced to be filled with variable and clause tiles,
respectively. Now consider the section in between.

For $k\geq1$ and $\l\geq0$,
consider a row of tiles $LW^kS^\l R$ (meaning an $L$ tile followed by a $W$ tiles
$k$ times, an $S$ tile $\l$ times, and ending with an $R$ tile).
The bottom color sequence is $01(02)^k(62)^\l63$.
One easily checks that the unique way to tile the next row is with $LW^{k-1}S^{\l+1}R$.

If $k=0$, we get the case where we have a row $LS^\l R$ with bottom color sequence $01(62)^\l63$.
The unique way to tile the next row is with $XB^\l K$.

The section below will be filled by filler tiles~$F$.
Thus every section below $c_i$ is filled uniquely, with the crossover
tile~$X$ occupying rows $i$ and $i+1$ in the first column.
\end{proof}

The above proof is illustrated with two examples in the next subsection.

\begin{cor}
The expression $E$ is satisfiable if and only if $\G_E$ is $\T'$-tileable.
Moreover, the reduction is parsimonious, that is, the number of tilings
of $\G_E$ is the number of satisfying assignments for $E$.
\end{cor}

The corollary follows immediately from the construction given above,
and concludes the proof of Lemma~\ref{lem-wang}.

\subsection{Examples of the tiling construction}
In Figure~\ref{diagram-1} we show how to place a crossover tile in a special case,
corresponding to expression $\{(x,y,x),(x,y,y)\}$.
We illustrate the crossings with a \emph{wiring diagram} and then give a complete Wang tiling.
In Figure~\ref{diagram-2} below we give a bigger example
of the wiring diagram and the unique Wang tiling, corresponding to expression
$\{(x,y,x),(x,y,z),(y,z,z)\}$.
\begin{figure}[hbtp]
   \subfloat[Wiring diagram]{\label{wiring-1}\includegraphics[scale=0.3]{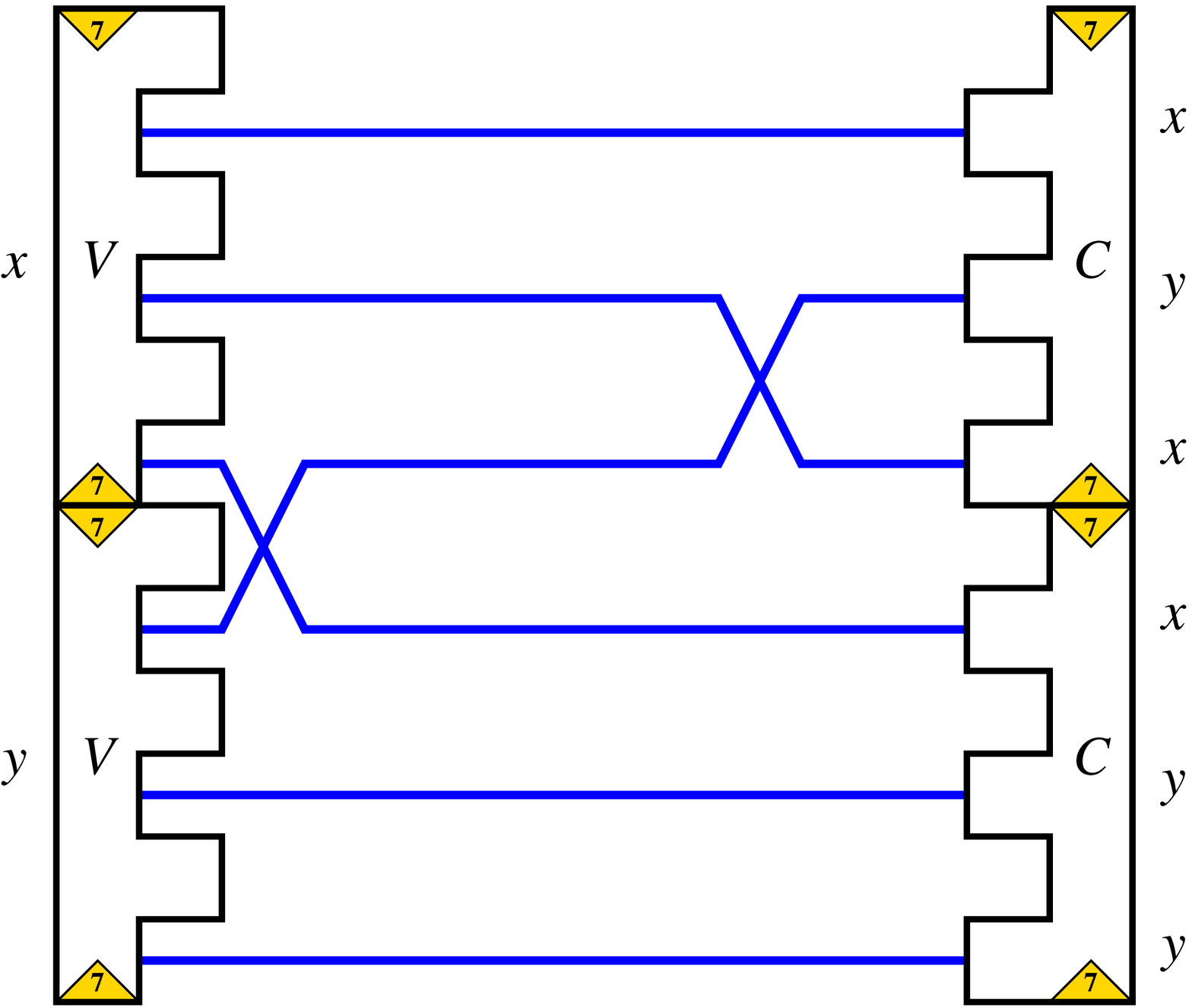}}
   \qquad
   \subfloat[Unique tiling]{\label{encode-1}\includegraphics[scale=0.3]{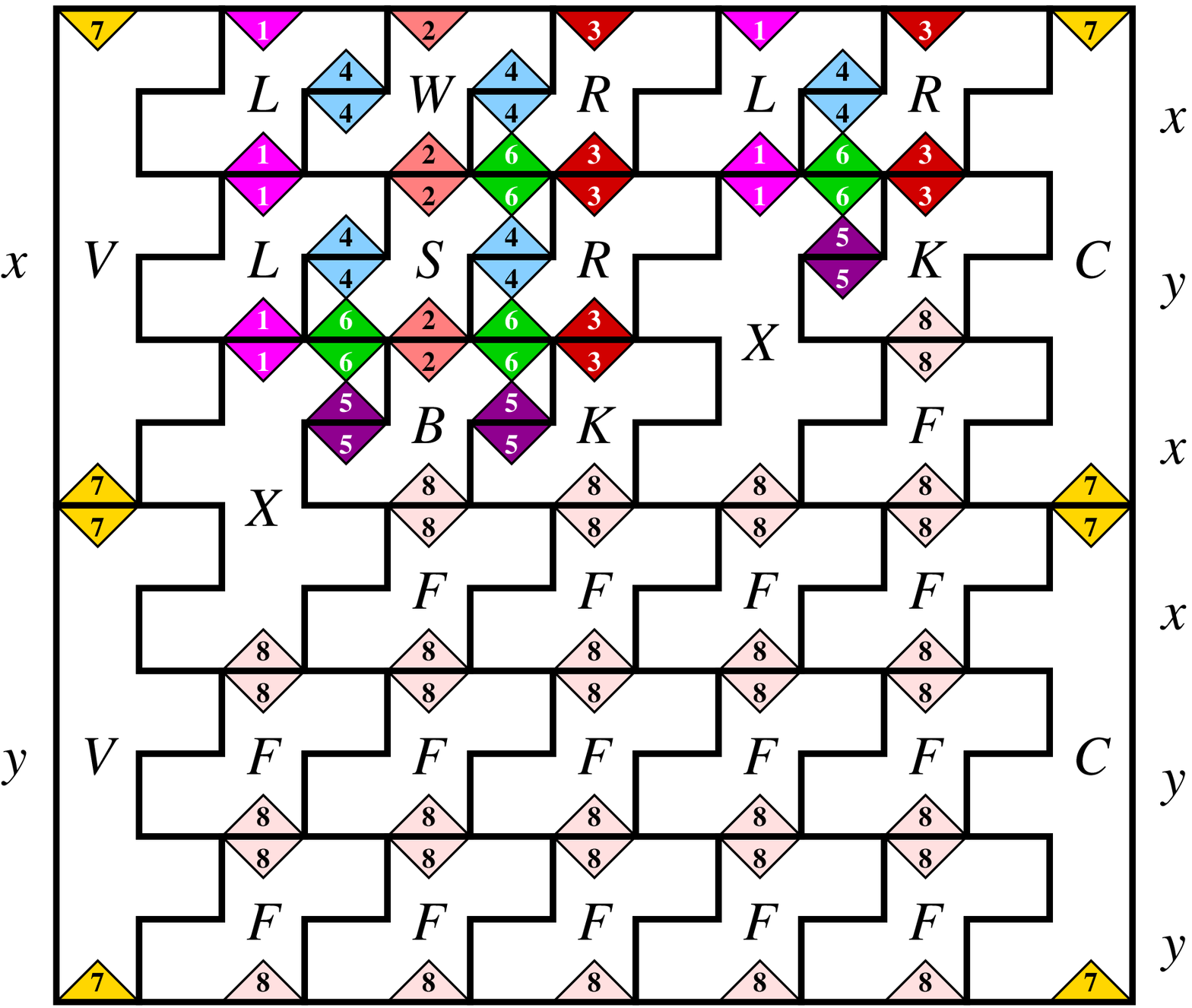}}
   \caption{A small example of how to place crossover tiles.}
   \label{diagram-1}
\end{figure}

\begin{figure}[hbtp]
   \subfloat[Wiring diagram]{\label{wiring-2}\includegraphics[scale=0.3]{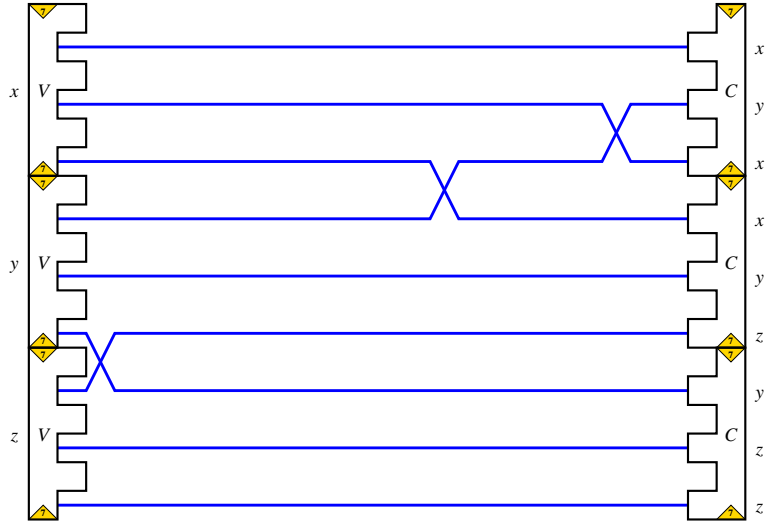}}
   \\
   \subfloat[Unique tiling]{\label{encode-2}\includegraphics[scale=0.3]{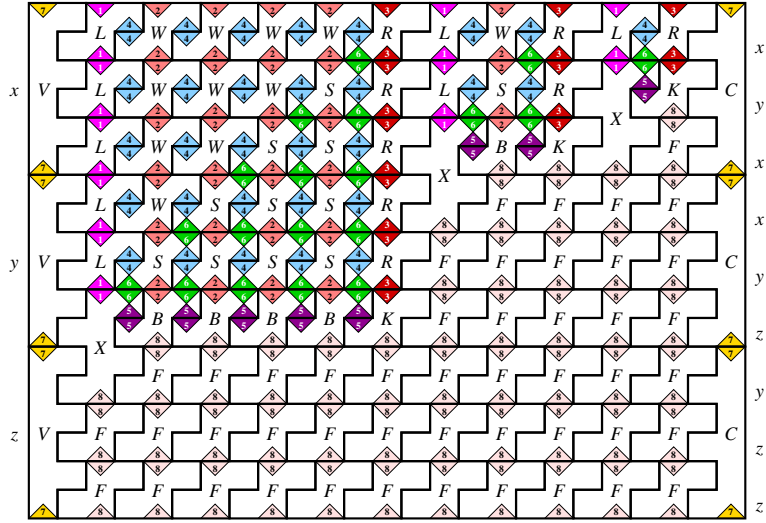}}
   \caption{A bigger example of the unique base tiling.}
   \label{diagram-2}
\end{figure}

\medskip

\section{Proof of the Second Reduction Lemma {\small(Lemma~\ref{lem-rect})}} \label{s:second-reduction}

\subsection{Basics}
In this section, we provide a further connection between Wang tiles and ordinary rectangular tiles
(by making a reduction from the latter to the former).
Recall that by Lemma~\ref{lem-reduction},
we can replace generalized Wang tiles with relational Wang tiles.

Without loss of generality, we may assume that the Wang relations are irreflexive,
that is, there is no tile $\t$ such that $H(\t,\t)$ or $V(\t,\t)$.
Indeed, suppose $\W$ is a set of Wang tiles.
Let $\W'=\{\t_i:i\in\{0,1\},\ \t\in\W\}$ be a doubled set of tiles.
Define its horizontal Wang relation as follows.
For $\t,\t'\in\W$ and $i,j\in\{0,1\}$,
let $H_{\W'}(\t_i,\t'_j)$ if and only if $H_{\W}(\t,\t')$ and $i\neq j$.
Its vertical Wang relation is defined analogously.
It is clear that the Wang relations of $\W'$ are irreflexive.
Moreover, a $\W$-tiling can be made into a $\W'$-tiling by adding subscripts to the tiles in a checkerboard fashion,
while the reverse can be done by ignoring the subscripts.
Of course, the same transformation is done on the boundary tiles as well.
Clearly this does not affect tileability nor the number of such tilings.

From now on, assume we are given a fixed set $\W$ of relational Wang tiles whose relations $H$ and $V$ are irreflexive.
Our goal is to produce a fixed set $\R$ of rectangular tiles with the following property:
Given any simply connected region $\G$ with specified boundary tiles,
we can produce (in linear time) a simply connected region $\G'$ such that $\G$
is $\W$-tileable if and only if $\G'$ is $\R$-tileable.  Moreover, the number of
$\W$-tilings of $\G$ will be the same as the number of $\R$-tilings of~$\G'$.

For simplicity, we first consider the case where we are given an $r\times c$
rectangular region $\G$ with specified boundary tiles.

\subsection{Expansion}
From this point on, we only consider tiling using rectangular tiles.
Fix $M$ and $\e$ to be positive integers.
Given a region $\G_0$, we obtain an \emph{$(M,\e)$-expansion} $\G$ by scaling $\G_0$ by a factor of $M$ and then
\emph{perturb} it by moving each corner vertex of the boundary curve of the region~$\G$, at most $\e$ in each
direction, such that $\G$ is still a region (with rectilinear edges).
Recall that a (rectangular) tile is just a simply connected region, thus the notion of $(M,\e)$-expansion of a tile is defined.
A tileset $\T$ is an \emph{$(M,\e)$-expansion} of a tileset $\T_0$ if each $\t\in\T$ is an $(M,\e)$-expansion of some $\t_0\in\T_0$.

A tiling $\p$ of a region $\G$ is an \emph{$(M,\e)$-expansion} of a tiling $\p_0$ of some region $\G_0$
if it can be obtained by dilating by a factor of $M$,
and then perturbing the tiles and the region by at most $\e$ as above.
Note that after scaling, each tile may grow or shrink in each dimension by at most $2\e$,
and can \emph{shift} around from its starting point by at most~$\e$.

Given a tileset $\T_0$ and an $(M,\e)$-expansion $\T$,
a region $\G$ \emph{respects the expansion} if there is a unique region $\G_0$ such that
any $\T$-tiling of $\G$ is an $(M,\e)$-expansion of a $\T_0$-tiling of~$\G_0$.

Intuitively, we will choose $M > 100\e$, say, and carefully perturb only a few tiles,
so that when consider tilings of regions respecting the expansion,
we can essentially predict what the new tiling can be based on the original tiling.

\newpage

\subsection{Rectangular tiles $\R_0$ and the region $\G_0(r,c)$} \label{group}
Consider the following tileset:
$$\R_0=\left\{f=R(34,11),\ w=R(31,14),\ s=R(10,10),\ h=R(11,31),\ v=R(14,34)\right\},$$
where $R(a,b)$ denotes a rectangle of height $a$ and width $b$ (see Figure~\ref{rect-tiles}).
For a rectangle $t$, write $\height t$ and $\width t$ for its height and width, respectively.
\begin{figure}[hbtp]
   \centering
   \subfloat
   {\includegraphics[scale=0.4]{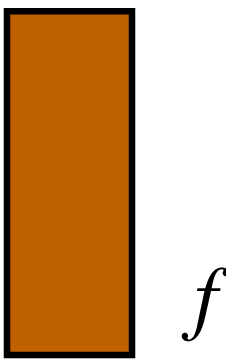}}
   \qquad\qquad
   \subfloat
   {\includegraphics[scale=0.4]{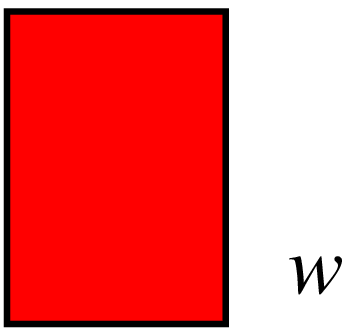}}
   \qquad\qquad
   \subfloat
   {\includegraphics[scale=0.4]{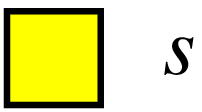}}
   \qquad\qquad
   \subfloat
   {\includegraphics[scale=0.4]{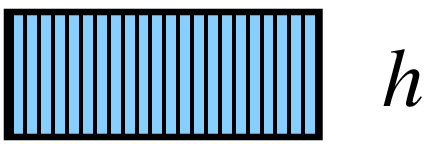}}
   \qquad\qquad
   \subfloat
   {\includegraphics[scale=0.4]{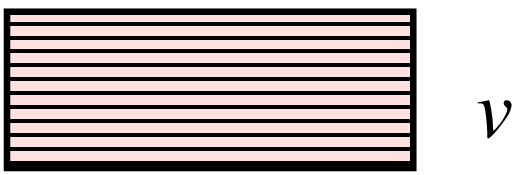}}
   \caption{Rectangular tiles $\R_0$.}
   \label{rect-tiles}
\end{figure}

Now consider the region $\G_0(r,c)$ defined as follows (see Figure~\ref{fig:boundary}).
On each vertical side, there are $r$ \emph{protrusions} of height $\height h$ and width $\width s$, separated by height $\height f$.
On each horizontal side, there are $c$ \emph{cavities} of width $\width v$ and height $\height s$, separated by width $\width f$.

\begin{figure}[hbtp]
   \includegraphics[scale=0.4]{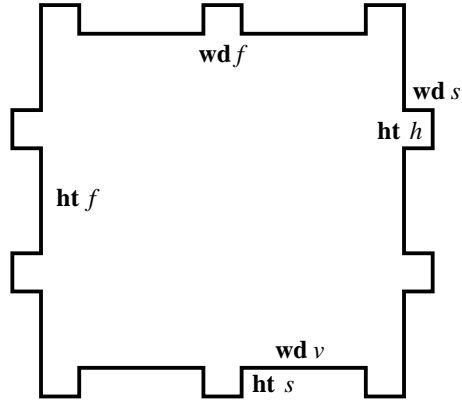}
   \caption{Boundary region $\G_0(2,2)$.}
   \label{fig:boundary}
\end{figure}

\begin{sublem}
\label{sl:unique-rect}
The unique $\R_0$-tiling of $\G_0(r,c)$ consists of $r$ rows and $c$ columns of the $w$ tile.
\end{sublem}
\begin{proof}
Fix natural numbers $a=10$ and $b=1$.
The tiles introduced above can now be written as $f=R(3a+4b,a+b)$, $w=R(3a+b,a+4b)$, $s=R(a,a)$, $h=R(a+b,3a+b)$, and $v=R(a+4b,3a+4b)$.

We begin with a few definitions.
A horizontal (vertical) segment of a region is called \emph{bounded} if the region extends downward (to the right) on both sides of the segment.
For $t\in\{v,h\}$, a \emph{pair $(t,s)$} is the configuration of placing the tile $s$ above or below $t$, aligned on the left.
The \emph{orientation} of the pair is positive (negative) if $s$ is placed below (above).
Similarly, for $t\in\{w,f\}$, a pair $(t,s)$ is obtained by placing $s$ to the left or right of $t$, aligned on top.
The orientation is positive (negative) if $s$ is placed to the right (left).
A bounded segment is \emph{tiled by a tile} (\emph{pair}) if in all tilings, the tile (pair) is adjacent to the segment.

We will tile the region $\G_0(r,c)$ in steps, as indicated by the numbers labeled on Figure~\ref{order}.
Note that since $a>b$, each bounded horizontal segment of width $\width f$ on the top border must be tiled by $f$ tiles,
labeled~1.
Similarly on the left, the bounded vertical segments of height $\height h$ must be tiled by $h$ tiles,
labeled~2.
This creates a bounded vertical segment of height $\height v+\height s$ on the top left corner;
since $a>3b$, it is tiled by the pair $(v,s)$, labeled~3.
Since $a>4b$, it is obvious that it needs to be positively oriented, to avoid a hole of width
$\width v-\width s$ and height $\height s$, which cannot be filled.

\begin{figure}[hbtp]
   \includegraphics[scale=0.6]{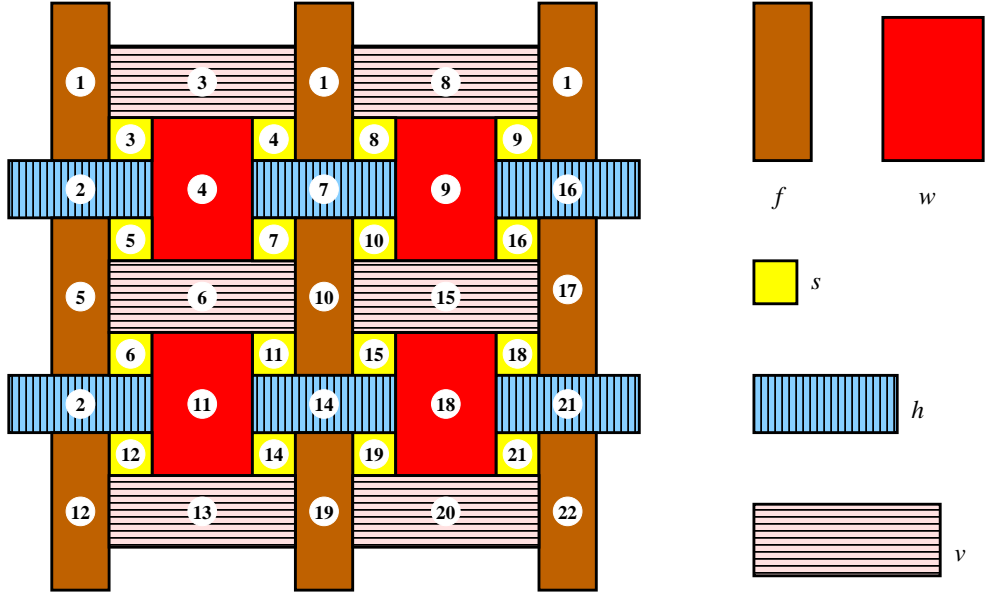}
   \caption{Unique base tiling labeled by order.}
   \label{order}
\end{figure}

Note that since $a>3b$, this creates a new bounded horizontal segment of width $\width w+\width s$, which is tiled
by the pair $(w,s)$, labeled~4.
If~$w$ is on the left, it will create a bounded horizontal segment of width $\width f+\width s$ to its left.
Otherwise, if $w$ is on the right, several $s$ will be forced to appear on the left and still create the same bounded segment.
Therefore, the $(w,s)$ pair creates the bounded segment, regardless of how it is oriented.

Since $a>3b$, this bounded horizontal segment of width $\width f+\width s$ is again tiled by an $(f,s)$ pair, labeled~5.
Like the $(v,s)$ pair above, since  $a>4b$, this needs to be positively oriented.
This creates the bounded vertical segment of height $\height v+\height s$, tiled by a pair $(v,s)$, labeled~6, as above.
In either orientation, it bounds the vertical segment of height $\height w$ above,
concluding that the $(w,s)$ pair (labeled 4) we placed above needs to be positively oriented.
Furthermore, this bounds the vertical segment of height $\height h+\height s$, again tiled by the pair $(h,s)$,
labeled~7.
As before, in either orientation, we have a bounded vertical segment of height $\height v+\height s$, which necessarily needs to be tiled by the positively oriented pair $(v,s)$, labeled~8.
This creates a bounded horizontal segment of width $\width w+\width s$.

We continue in like manner, working our way on the anti-diagonal from top right to bottom left.
Each time we place the pair $(w,s)$, forcing the adjacent pair $(h,s)$ placed in the previous stage to be positively oriented.
Then we place $(f,s)$, forcing the adjacent $(v,s)$ to be positively oriented as well.
This procedure repeats with $(w,s)$ and $(f,s)$ in an alternating fashion.
The last $(f,s)$ will be placed in positive orientation and creates a bounded vertical segment of height $\height v+\height s$.

Similarly, we work from bottom left to top right on the next anti-diagonal.
We alternate between placing $(v,s)$ and $(h,s)$ pairs, positively orienting the $(w,s)$ and $(f,s)$ pairs in the previous stage, respectively.
This continues until the entire region is filled.
\end{proof}

\subsection{Expansion $\R$ of $\R_0$}
We will now define a clever set of perturbed expansion tiles that will correspond to the relational Wang tiles.
Only the tiles $s$, $h$, and $v$ will have perturbations.
Let $\W=\{\t_1,\ldots,\t_n\}$ be the fixed set of relational Wang tiles with irreflexive horizontal and vertical Wang relations $H$ and $V$, respectively.
Fix $\e=5^n$ and $M=100\e$ for the remainder of the section.
Let $\R$ be an $(M,\e)$-expansion of $\R_0$ as follows:

For $t\in\{s,h,v\}$, let $t(a,b)$ be the scaled version of $t$ with height and width increased by $a$ and $b$, respectively.
Imagine that the $h$ and $v$ tiles can stretch horizontally and vertically, respectively,
and the $s$ tiles can stretch in both directions.
Then the $w$ tiles, having no perturbations, will only shift around a little (by at most $\e$).
The $f$ tiles will stay fixed, enforcing the global structure.
See Figure~\ref{fig:shifting}.
A $w$ tile will be shifted to the right and down by $5^i$ to \emph{represent} the Wang tile $\t_i$.
To restrict the shifts to only those sizes,
we replace $s$ with the appropriate perturbed versions.
Namely, for each $i$, introduce four tiles with perturbations $s(\pm 5^i,\pm 5^i)$, where all four combinations of signs are included.
To enforce the Wang relations,
for each $\t_i,\t_j\in\W$ such that $V(\t_i,\t_j)$ or $H(\t_i,\t_j)$, we introduce the perturbation $v(5^j-5^i,0)$ or $h(0,5^j-5^i)$, respectively.
This is the set $\R$ we will use.

\begin{figure}[hbtp]
   \subfloat{\includegraphics[scale=0.33]{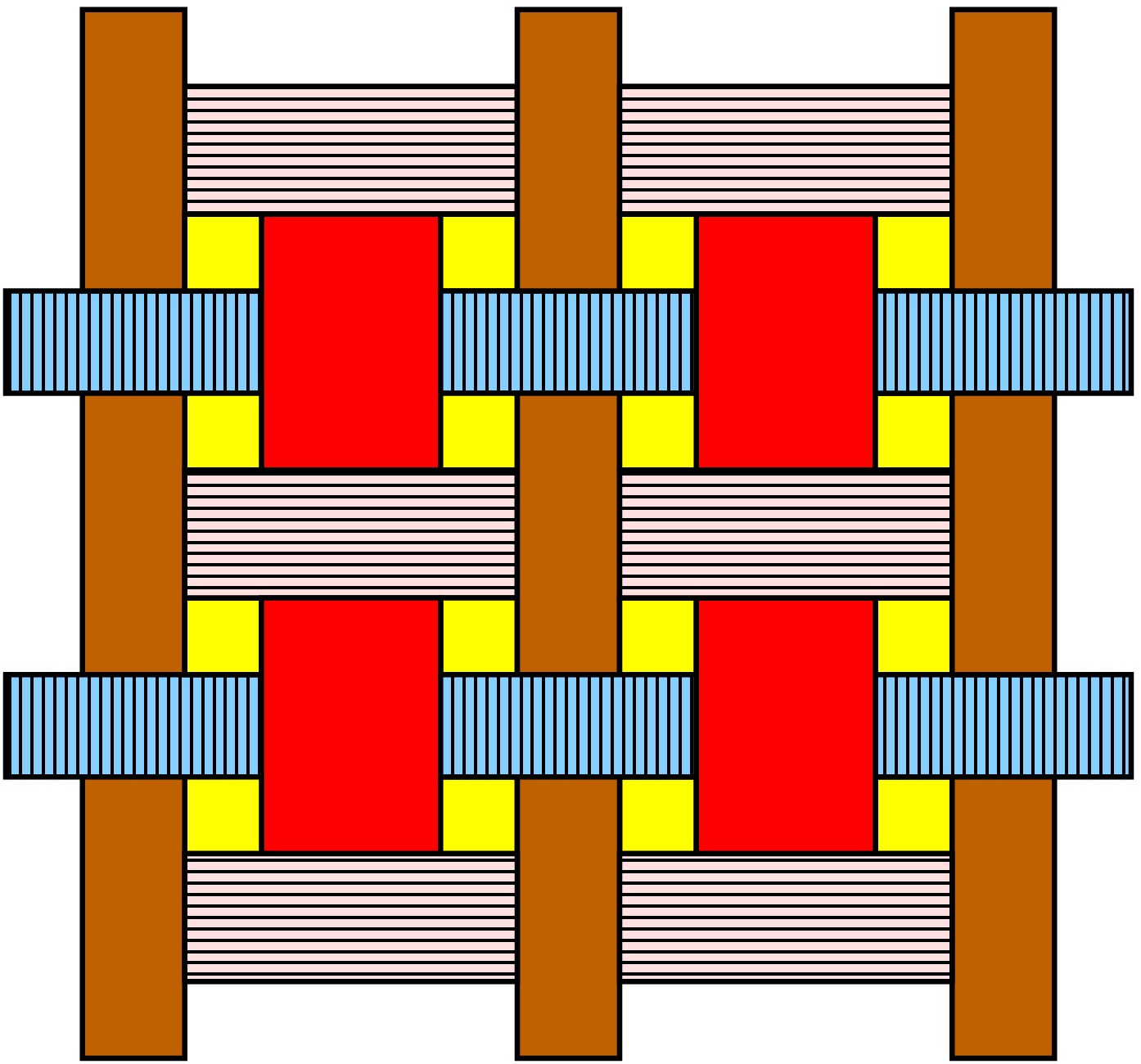}}
   \qquad
   \subfloat{\includegraphics[scale=0.33]{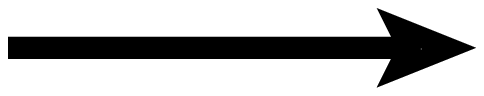}}
   \qquad
   \subfloat{\includegraphics[scale=0.33]{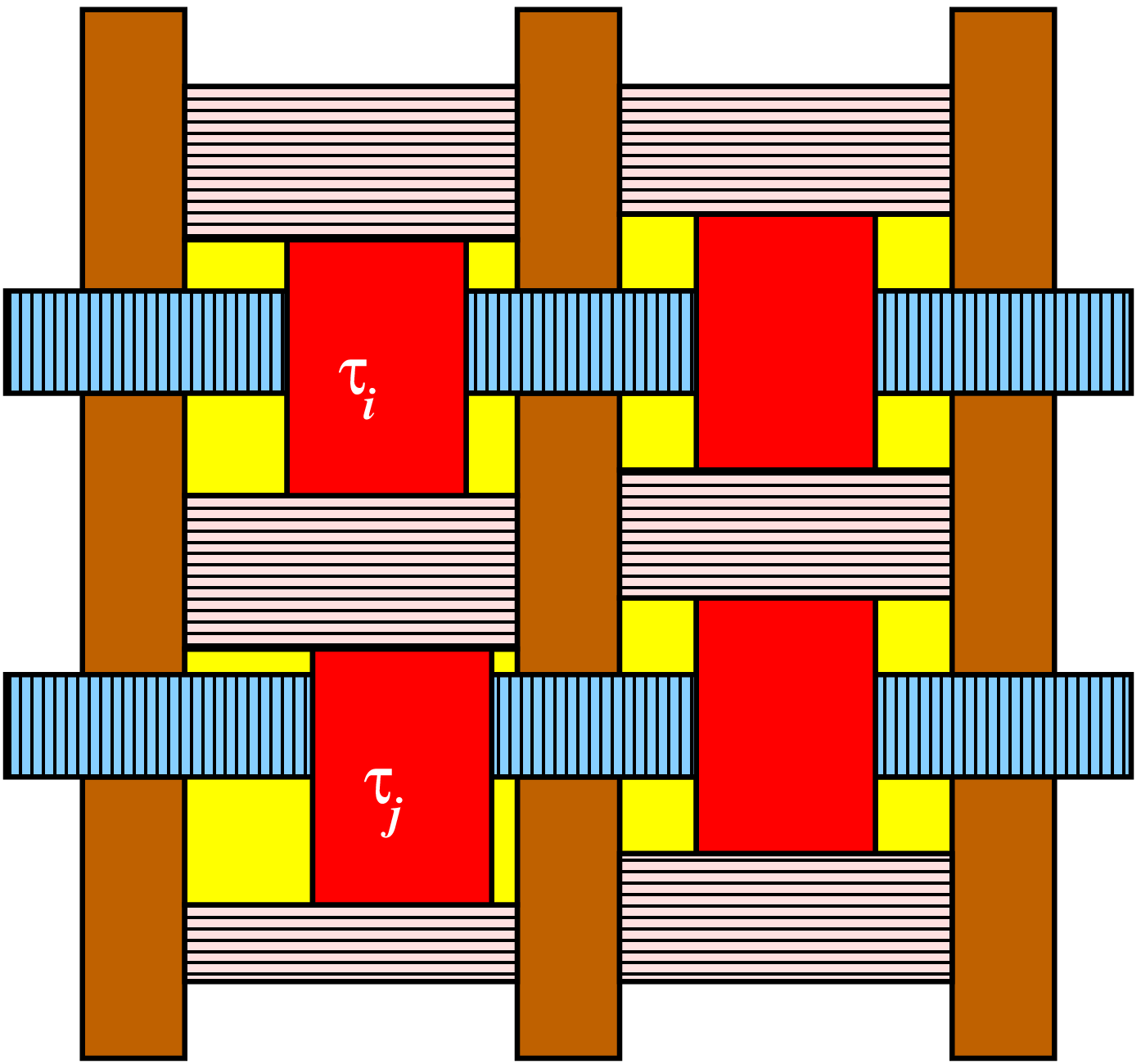}}
   \caption{Shifting an expansion of the unique tiling to represent Wang tiles.}
   \label{fig:shifting}
\end{figure}

\subsection{Rectangular tiling}
Obtain an $(M,\e)$-expansion $\G(r,c)$ of $\G_0(r,c)$ by scaling with a factor of $M$ and then perturbing it as follows.
Recall that there are $r$ protrusions on each vertical side and $c$ cavities on each horizontal side.
Each protrusion or cavity corresponds to a boundary tile of $\G$ in a natural way.
Perturb the protrusion or cavity to the right or down, respectively, by $5^i$ units if it corresponds to $\t_i$.

\begin{sublem}
The $(M,\e)$-expansion $\G(r,c)$ of $\G_0(r,c)$ respects the expansion $\R$ of~$\R_0$.
\end{sublem}
\proof
Recall the argument in the proof of Sublemma~\ref{sl:unique-rect}.
As the inequalities are all satisfied,
the $f$ tiles are fixed and force the perturbations to stay local.
The $w$ tiles have two degrees of freedom.
They can move $\pm5^i$ in each direction, as regulated by the $s$ tiles.
Now note that the inequalities in the proof of Sublemma~\ref{sl:unique-rect} are preserved.
We leave the (easy) details to the reader.
\qed

\medskip

We now return to the proof of Lemma~\ref{lem-rect}.
It is clear that given a Wang $\W$-tiling of the rectangle $\G$ with boundary, we will get an $\R$-tiling of $\G(r,c)$.
Indeed, simply take the unique tiling of $\G_0(r,c)$ as afforded by Sublemma~\ref{sl:unique-rect},
scale by a factor of $M$,
and then shift each $w$ tile to the right and down by $5^i$ if it represents $\t_i$,
and adjust the other tiles in the obvious way.

Conversely, if we are given an $\R$-tiling of $\G(r,c)$, we wish to recover the $\W$-tiling of $\G$.
This is achieved using the following two sublemmas,
both of which are clear when all numbers are considered in base $5$;
we omit the (easy) details.

\begin{sublem}
The equation $5^i-5^j=5^k+5^\l$ does not admit a solution in $\N$.
\end{sublem}
Therefore each $w$ tile will shift to the right and down (as opposed to shifting left or up),
and hence indeed represents a Wang tile $\t_i$ for some $i$.

\begin{sublem}
The equation $5^i-5^j=5^k-5^\l$ does not admit solutions in $\N$ except if $i=j$ or $i=k$.
\end{sublem}

If a $w$ tile representing $\t_j$ is to the right of a $w$ tile representing $\t_i$,
then $h(0,5^j-5^i)$ must be in $\R$.
By the sublemma above, the differences $5^j-5^i$ are all distinct (recall that the Wang relations are irreflexive, so $i=j$ does not happen),
therefore we must have had $H(\t_i,\t_j)$ as part of the Wang relation.
Similarly for the vertical Wang relation $V$.
So by reading off the associated tile $\t_i$ from the shifts of each $w$ tile,
we get a Wang $\W$-tiling of~$\G$.

This completes the construction of~$\G_0(r,c)$ for the case when $\G$ is a rectangle.
For the general case, when $\G$ is a simply connected region, the proof follows
verbatim after replacing $\G(r,c)$ and $\G_0(r,c)$ by appropriate regions.

It remains to get the upper bound estimates on the number of rectangles involved in the construction.
Suppose we are given a set of $k$ ordinary Wang tiles using $c$ colors (on the boundary).
By Lemma~\ref{lem-reduction} we can equivalently consider a set of less than $k+4c$ relational Wang tiles.
To satisfy irreflexivity, we might need to double the set of tiles,
resulting in $n=\abs{\W}<2(k+4c)$ tiles.
When making $\R$, we will have one each of $f$ and $w$ tiles.
There will be $4n$ perturbed $s$ tiles and at most $n^2$ perturbed $h$ and $v$ tiles each.
In total, $$\abs{\R}\leq 2n^2+4n+2=2(n+1)^2\leq 8(k+4c)^2.$$
This concludes the proof of Lemma~\ref{lem-rect}.

\medskip

\section{Proof of theorems}\label{s:proofs}

\subsection{Proof of Theorem~\ref{thm-rect}}\label{ss:proof-thm-rect}
In the proof of Lemma~\ref{lem-wang} in Section~\ref{s:first-reduction},
we constructed the set $\W$ of $23$ generalized Wang tiles using~$9$ colors, such that \sct{\W} is \NP-complete.
It remains to count the total number of rectangles we obtain from the series of reduction constructions.

First, we compute the number of ordinary Wang tiles given by the transformation
in Lemma~\ref{lem-reduction}.  Observe that the total area of tiles in~$\W$ is
$9\cdot5+8\cdot4+4\cdot14=133$.
Therefore we can break them into $133$ ordinary Wang tiles by adding $133-23$ more colors.
But as these colors do not appear on the boundary, they need not be counted.
Hence, in Lemma~\ref{lem-rect}, we can take $k=133$ and $c=9$, thus giving us at
most $10^6$ rectangles.
\qed

\subsection{Proof of Theorem~\ref{thm-sharp}}\label{ss:proof-thm-sharp}
First, note that the reduction in the proof of Theorem~\ref{thm-rect} is parsimonious.
However, there seems to be no \SP-completeness result for the
\problem{\#Cubic Monotone $1$-in-$3$ SAT} problem.  This
is easy to fix by making a similar reduction from the \problem{2SAT} problem,
whose associated counting problem is \SP-complete (see \cite{Val}).

An instance of \problem{2SAT} is a set of variables and a collection of clauses.
Each clause is a disjunction of two literals, where each literal is either a variable or a negated variable.
The problem is to decide whether there is a satisfying assignment such that each clause has at least one true literal.
We modify the proof of Lemma~\ref{lem-wang} to obtain a parsimonious reduction from \problem{2SAT}.
By replacing the two variations of the variable tile by the ones shown in Figure~\ref{fig:2sat-v},
we may set up unnegated and negated copies of a single variable.
Indeed, with a sequence of $5(26)^{r-1}36(26)^{s-1}4$ as colors on the left vertical edge,
we create a list of $r+s$ variables, where the last $s$ are negated.
By replacing the three variants of the clause tile by the three obvious candidates in Figure~\ref{fig:2sat-c},
we force each clause to be satisfied.

\begin{figure}[hbtp]
   \centering
   \subfloat[tiles for variables]{\label{fig:2sat-v}\includegraphics[scale=0.4]{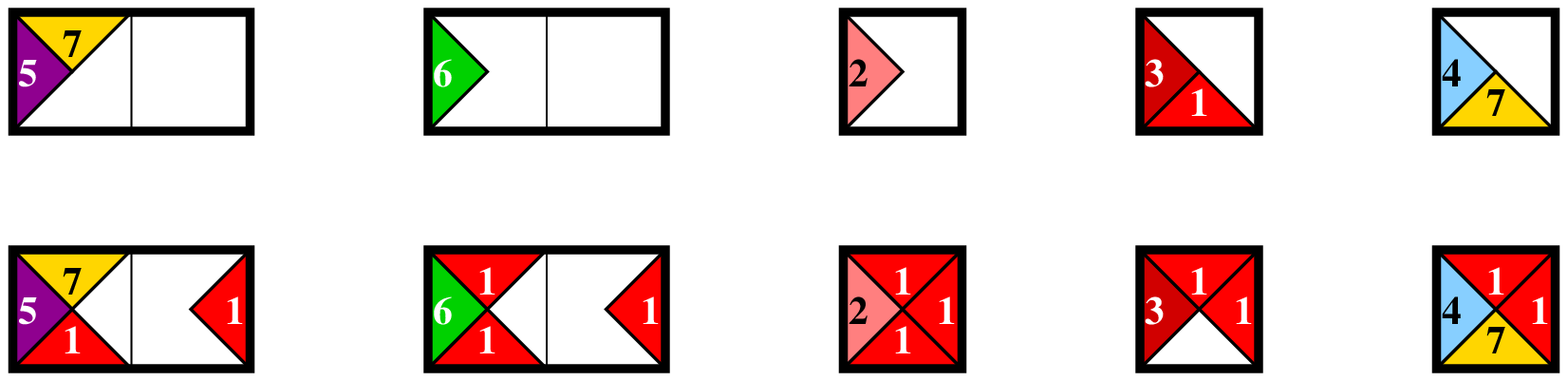}}
   \qquad\qquad
   \subfloat[clause tiles]{\label{fig:2sat-c}\includegraphics[scale=0.4]{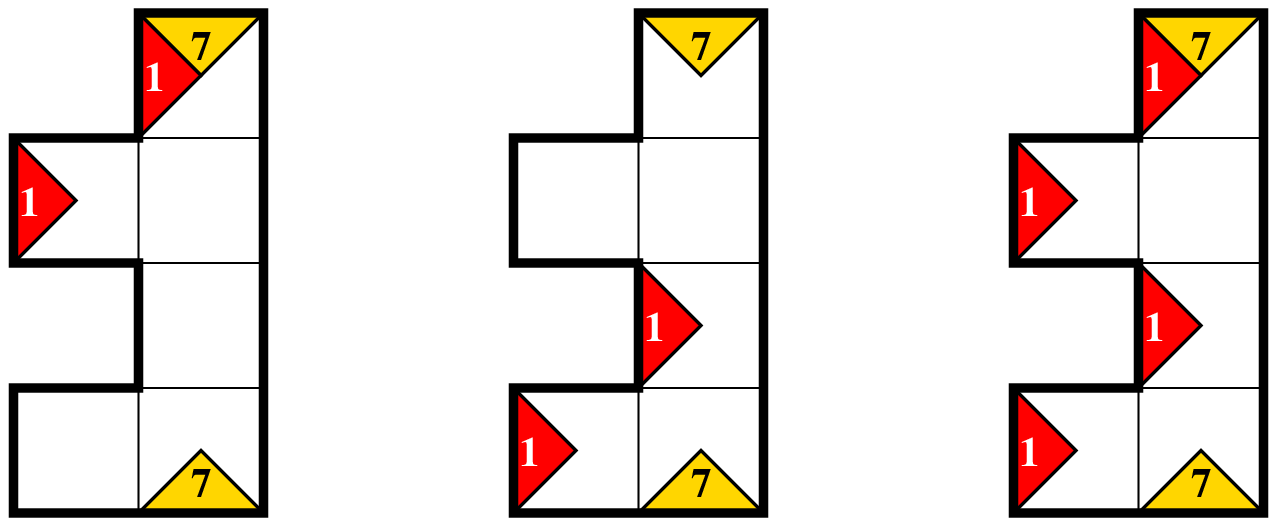}}
   \caption{Tiles for \problem{2SAT}.}
   \label{fig:2sat}
\end{figure}

Note that the modified tileset has a smaller total area, and has the same number of colors used on the boundary.
Therefore as in the proof of Theorem~\ref{thm-rect}, we apply Lemma~\ref{lem-rect} to conclude that $10^6$ rectangles suffice.
\qed


\medskip

\section{Final remarks and open problems}\label{s:fin}

\subsection{}\label{ss:fin-levin}
Theorem~\ref{thm-2-dim} was only announced in~\cite{GJ},
referencing an unpublished preprint.  Of course, now we have 
much stronger results.

A version of Lemma~\ref{lem-wang} was first announced in Levin's original
1973 short note regarding \NP-completeness~\cite{Lev}, but the proof has never been 
published.\footnote{Leonid Levin, personal communication.}
Although we were unable to find in the literature an explicit 
construction for either Lemma~\ref{lem-wang} or, equivalently, of 
Theorem~\ref{thm-ordinary}, we do not claim this result as ours, since
it became a folklore decades ago.  We include the proof for completeness,
and since we need an explicit construction.  An alternative proof is 
outlined in Subsection~\ref{ss:fin-moore} below.

Let us mention that using~\cite{Oll}, the number of tiles in 
Theorem~\ref{thm-ordinary} can be reduced to~$11$, 
but this reduction has no effect on the number of tiles in the main theorems.
Indeed, Theorem~\ref{thm-ordinary} is an immediate corollary of Lemma~\ref{lem-wang},
which is the one needed in the proof of main 
theorems~\ref{thm-rect} and~\ref{thm-sharp}.

\subsection{}\label{ss:fin-moore}
Our proof of Lemma~\ref{lem-wang} is completely elementary
and yields explicit bounds
(see also Subsection~\ref{ss:fin-levin}).
Let us sketch an alternate proof of the lemma,
using a non-deterministic \emph{universal Turing machine} (UTM).
It was suggested to us by Cris Moore.

Fix some non-deterministic universal Turing machine~$\M$.
Given two finite tape configurations and a natural number~$t$ (in unary),
it is \NP-complete to decide whether~$\M$ transforms the first tape
configuration to the second with~$t$ steps of computation.
Fix a finite set $\W$ of Wang tiles that simulate the space-time
computation diagram of~$\M$ (see \latin{e.g.}~\cite[\S7]{LeP}).
Encode the given tape configurations as the top and bottom
boundaries of a rectangular region with height~$t$.
This region is tileable by $\W$ if and only if~$\M$ transforms
the first tape configuration to the second in precisely~$t$ steps.
The details are straightforward.

Note that this method also proves the counting result.
Indeed, one can devise a UTM so that
there is a bijective correspondence between the accepting paths of
the UTM and of the Turing machine it is simulating.

\smallskip

The proof of Lemma~\ref{lem-wang} constructs a set of $23$ generalized Wang tiles ($133$ ordinary Wang tiles).
However, it is possible to decrease these numbers by elementary means.
After this paper was written, a modified construction by G\"{u}nter Rote
and the second author improves the number of generalized Wang tiles in Lemma~\ref{lem-wang} to~$15$,
which amounts to $35$ ordinary Wang tiles.
With other technical improvements this does reduce the $10^6$ bound in Theorem~\ref{thm-rect}
to a much friendlier~$117$.  The details are given in~\cite{Yang}.

We do not know if this approach leads to improvements in the number of Wang tiles in the lemma,
as this would depend on the smallest UTM.
Given an $m$-state $n$-symbol Turing machine with $k$ instructions,
the standard construction of Wang tiles to simulate such a Turing machine
yields more than $nm+n+k$ tiles.
As a perspective, among the smallest known UTMs,
this minimum is achieved by Rogozhin's $4$-state $6$-symbol machine with $22$ instructions,
which already yields more than $52$ tiles~\cite{Rog} (see also~\cite{NW}).
Unless a substantial progress is made in finding small UTMs, our elementary proof still gives better bounds.

\subsection{}
In the tiling literature, the original theoretical emphasis was
on tileability of the plane, the decidability and aperiodicity.
The problem was often stated in the equivalent language
of Wang tiles~\cite{Ber,Rob,Wang}.  Unfortunately, there does
not seem to be any standard treatment of the
\emph{finite Wang tiling} problems.  Although some equivalences
in the Lemma~\ref{lem-reduction} are routine, such as the reduction
in Figure~\ref{iii-iv}, others seem to be new.  We present full
proofs for completeness.

\subsection{}
Historically, finite tilings were a backwater of the tiling theory,
with coloring arguments being the only real tool~\cite{Gol1}.
On a negative (complexity) side, originally, the tileability
problem was studied for general regions, where the tiles were
part of the input.  The \NP-completeness of this most general
problem is given in~\cite[$\S$GP13]{GJ}.   When the set of tiles
is fixed, \NP-completeness was shown for general regions and
various fixed small sets of tiles (see~\cite{MR} and~\cite{BNRR}
building on the earlier unpublished work by Robson).

On the positive side, papers of Thurston~\cite{Thu} and
Conway \& Lagarias~\cite{CL} introduced the \emph{height function}
and the \emph{tiling group} interrelated approaches.
The key underlying idea is the use of combinatorial group theory applied
to the boundary word of the \emph{simply connected regions}, so the tilings
become \emph{Van Kampen diagrams} of the corresponding tiling group.
This approach allowed numerous applications to perfect matchings~\cite{Cha},
tile invariants~\cite{Korn,MP,Reid:homotopy}, tileability~\cite{She},
various local move connectivity results~\cite{KP,Rem1},
classical geometric problems~\cite{Ken}, applications
to colorings and mixing time~\cite{LRS}, \latin{etc}.  More relevant to this
paper, the breakthrough result by C.~and R.~Kenyon~\cite{KK} proved that
tileability with bars of simply connected (s.c.) regions can be decided in
polynomial time.  This result was further extended to all pairs of
rectangles by R\'{e}mila in~\cite{Rem2}, and by Korn~\cite{Korn} to
an infinite family of \emph{generalized dominoes}.  Our Main Theorem
puts an end to the hopes that these results can be extended to
larger sets of rectangles.

Note also that having s.c.~regions gives a speed-up for polynomial
problems.  For example, domino tileability is a special case of
perfect matching, solvable in quadratic time on all planar bipartite
graphs~\cite{LP}.  However, Thurston's algorithm is linear time
(in the area), for all s.c.~regions (see~\cite{Cha,Thu}).

\subsection{}\label{ss:fin-three-rect}
We conjecture that in the Main Theorem (Theorem~\ref{thm-rect}), the number of rectangles
can be reduced down to~3, thus matching the lower bound (R\'{e}mila's tileability algorithm
for the case of two rectangles).  As a minor supporting evidence in favor of this conjecture,
let us mention that the proofs in~\cite{KK,Rem2} are crucially based on
\emph{local move connectivity}, which fails for three general rectangles.  In the absence of
algebraic methods, there seem to be no other (positive) approach to tileability.

\subsection{}
This result of Main Theorem can be contrasted with a large body
of positive results on tiling rectangular regions with a fixed set
of rectangles.

\begin{thm}[``Tiling rectangles with rectangles'' Theorem~\cite{LMP}]
\label{thm-lmp}
For every finite set $\R$ of rectangular tiles, the tileability problem
of an $\ts [M\times N]$ rectangle can be decided in $O(\log M + \log N)$ time.
\end{thm}

Note that Theorem~\ref{thm-lmp} has linear time complexity for
the rectangular regions written \emph{in binary}.
This result is based on the pioneer results by Barnes~\cite{Bar1,Bar2}
applying commutative algebra, the \emph{finite basis theorem}~\cite{DK}
(see also~\cite{Reid:klarner}), and the \emph{transfer matrix method}
(see \latin{e.g.}~\cite[Ch.~4]{Sta}).

It seems, tilings of rectangles have additional structure, which general
regions do not have.  See \latin{e.g.}~\cite{BSST,C+,Rob0} for assorted results
on the subject.  On the other hand,
when the tiles are part of the input,
deciding tileability can be \NP-hard, and the proof can be used to show
that counting tilings is \SP-hard.  Note that the results in~\cite{LMP}
only discuss tileability, not counting.  It would be interesting to
obtain general results on the local move connectivity and hardness of
counting results for tilings of rectangular regions with rectangles.

\subsection{}\label{ss:fin-two-rect-count}
Although counting perfect matchings in general graphs is
\SP-complete, for the grid graphs a \emph{Pfaffian formula} gives
a count for the number of domino tilings for any (not necessarily
simply connected) region; this formula can be applied in polynomial
time~\cite{LP} (see also~\cite{Ken-dimer}).
In a different direction, Moore and Robson~\cite{MR} conjecture that
already for two bars, the problem is \SP-complete  for general regions.
They note that the corresponding reductions in~\cite{BNRR,MR} are
not parsimonious.  Thus, until now, the \SP-completeness was open
for any finite set of rectangular tiles, even for general regions.

We make a stronger conjecture that for every tileset $\T$ of two bars $[1\times k]$ and
$[\ell \times 1]$, where $k, \ell \ge 2$, $(k,\ell) \ne (2,2)$, the
counting of tilings by~$\T$ of simply connected regions is \SP-complete.
In particular, the number~$10^6$ in Theorem~\ref{thm-sharp} can be
decreased to~2.  There is no direct evidence in favor of this,
except that the general combinatorial counting problems tend to be
\SP-complete unless there is a special algebraic formula counting them.
Furthermore, when it comes to tile counting, there seem to be no
direct benefit of simple connectivity of the regions, so such result
is likely to be equally hard as for general regions.
We refer to~\cite{Jer} for the introduction and references.

\subsection{}
By a simple modification of the Wang tiles, we can also get a
parsimonious reduction from~\problem{SAT}.
For that, first, we can introduce wire splitters and the NOT gate.
By doing so, we remove the ``cubic''  and ``monotone''
constraints, respectively.  These would play the same role as
crossover tiles, and require a separate color on the boundary for each.
This would also increase the set of tiles by introducing new variants
for the~$V$ and~$L$ tiles as well.  We omit the details.

We can then introduce the AND gate in a similar fashion, again with
a new control color on the top and new versions of the~$V$,~$C$
and~$L$ tiles.  This gives the embedding of \problem{SAT}.
This reduction is parsimonious in the same way as the reduction
in Theorem~\ref{thm-sharp}, which implies that the associated
counting problem is also \SP-complete.

Let us compute the total number of rectangles necessary for this
construction.   First, this would increase the number of Wang tiles from
$23$ to no more than $23\cdot 8$.  Then, the same argument as
above gives the $10^{8}$ bound in the number of rectangular tiles.
We omit the (easy) calculation and details.

\subsection{}\label{ss:fin-uniqueness}
The reductions in this paper can be used to prove \emph{uniqueness} results on
tileability with rectangles, i.e.~whether there exists a unique tiling
of a region with a given set of rectangular tiles. In~\cite{BNRR}, the problem
was completely resolved in the case of two bars.  An even simpler solution
follows from~\cite{KK} in this case.
Since all tilings are local move connected, taking the ``minimal tiling''
constructed by the algorithm in~\cite{KK} and trying all potential moves
gives an easy polynomial time test.
More generally, R\'{e}mila~\cite{Rem2} showed that for two general
rectangles one can go from one to another with certain non-local moves which
are easy to describe.  Again, since he produces the ``minimal tiling,'' his
algorithm can be used to decide unique tileability with two rectangles.

Now, our approach, via reduction from the general \problem{SAT} problem
(see above) shows that for a certain
finite set of rectangles, uniqueness of tilings of a simply connected
region is as hard as \problem{UNIQUE SAT}, which is \co-NP-hard and
has been extensively studied~\cite{BG,VV}.  This seems to be the first
result of this type.

\subsection{}\label{ss:fin-high-dim}
Although Theorem~\ref{thm-lmp} extends directly to
bricks in higher dimensions~\cite{LMP}, this is an exception rather
than the rule.  In fact, we recently showed that almost no other
positive tileability results extend to higher dimensions, even
Thurston's algorithm mentioned above (see~\cite{PY-domino}).

\vskip.7cm


\noindent
\textbf{Acknowledgements.} \,
We are grateful to Alex Fink, Jeff Lagarias, Leonid Levin, Cris Moore, G\"{u}nter Rote, and Damien Woods
for helpful conversations at various stages of this project.
We also thank the anonymous referees for attentive reading and useful comments on previous versions of this paper.
The first author is partially supported by the NSF and BSF grants.
The second author is supported by the NSF under Grant No.~DGE-0707424.


\newpage

\begin{thebibliography}{BNRR}

\bibitem[Bar1]{Bar1}
F.~W.~Barnes,
Algebraic theory of brick packing~I,
\emph{Discrete Math.}~\textbf{42} (1982), 7--26.

\bibitem[Bar2]{Bar2}
F.~W.~Barnes,
Algebraic theory of brick packing~II,
\emph{Discrete Math.}~\textbf{42} (1982), 129--144.

\bibitem[BNRR]{BNRR}
D.~Beauquier, M.~Nivat, \'{E}.~R\'{e}mila and M.~Robson,
Tiling figures of the plane with two bars,
\emph{Comput. Geom.}~\textbf{5} (1995), 1--25.

\bibitem[Ber]{Ber}
R.~Berger,
The undecidability of the domino problem,
\emph{Mem. AMS}~\textbf{66} (1966), 72~pp.

\bibitem[BG]{BG}
A.~Blass and Y.~Gurevich,
On the unique satisfiability problem,
\emph{Inform. and Control}~\textbf{55} (1982), 80--88.

\bibitem[BSST]{BSST}
R.~L.~Brooks, C.~A.~B.~Smith, A.~H.~Stone and W.~T.~Tutte,
The dissection of rectangles into squares,
\emph{Duke Math. J.}~\textbf{7} (1940), 312--340.

\bibitem[Cha]{Cha}
T.~Chaboud,
Domino tiling in planar graphs with regular and bipartite dual,
\emph{Theor. Comp. Sci.}~\textbf{159} (1996), 137--142.

\bibitem[C+]{C+}
F.~R.~K.~Chung, E.~N.~Gilbert, R.~L.~Graham, J.~B.~Shearer and J.~H.~van~Lint,
Tiling rectangles with rectangles,
\emph{Math. Mag.}~\textbf{55} (1982), no.~5, 286--291.

\bibitem[CL]{CL}
J.~H.~Conway and J.~C.~Lagarias,
Tilings with polyominoes and combinatorial group theory,
\emph{J. Comb. Theory, Ser.~A}~\textbf{53} (1990), 183--208.

\bibitem[CH]{CH}
N.~Creignou and M.~Hermann,
On \SP-completeness of some counting problems,
Research Report~2144, INRIA, 1993.

\bibitem[DK]{DK}
N.~G.~de~Bruijn and D.~A.~Klarner,
A finite basis theorem for packing boxes with bricks,
\emph{Philips Res. Rep.}~\textbf{30} (1975), 337$^\ast$--343$^\ast$;
available at \ts \url{http://tinyurl.com/65g8kvr}

\bibitem[GJ]{GJ}
M.~Garey and D.~S.~Johnson,
\emph{Computers and Intractability: A Guide to the Theory of \NP-completeness},
Freeman, San Francisco, CA, 1979.

\bibitem[Gol1]{Gol1}
S.~Golomb,
\emph{Polyominoes},
Scribners, New York, 1965.

\bibitem[Gol2]{Gol2}
S.~Golomb,
Tiling with sets of polyominoes,
\emph{J. Combin. Theory}~\textbf{9} (1970), 60--71.

\bibitem[Gon]{Gon}
T.~F.~Gonzalez,
Clustering to minimize the maximum intercluster distance,
\emph{Theor. Comp. Sci.}~\textbf{38} (1985), 293--306.

\bibitem[GS]{GS}
B.~Gr\"{u}nbaum and G.~C.~Shephard,
\emph{Tilings and patterns},
Freeman, New York, 1987.

\bibitem[Jer]{Jer}
M.~Jerrum,
\emph{Counting, sampling and integrating: algorithms and complexity},
Birkh\"{a}user, Basel, 2003.

\bibitem[Ken1]{Ken}
R.~Kenyon,
A note on tiling with integer-sided rectangles,
\emph{J. Combin. Theory, Ser.~A}~\textbf{74} (1996),
321--332.

\bibitem[Ken2]{Ken-dimer}
R.~Kenyon,
An introduction to the dimer model,
in \emph{ICTP Lect. Notes~XVII}, Trieste, 2004.

\bibitem[KK]{KK}
C.~Kenyon and R.~Kenyon,
Tiling a polygon with rectangles,
in \emph{Proc.~33rd FOCS} (1992), 610--619.

\bibitem[Korn]{Korn}
M.~Korn,
\emph{Geometric and algebraic properties of polyomino tilings},
MIT Ph.D. thesis, 2004;
available at \ts \url{http://dspace.mit.edu/handle/1721.1/16628}

\bibitem[KP]{KP}
M.~Korn and I.~Pak,
Tilings of rectangles with T-tetrominoes,
\emph{Theor. Comp. Sci.}~\textbf{319} (2004), 3--27.

\bibitem[LMP]{LMP}
T.~Lam, E.~Miller and I.~Pak,
Tiling rectangles with rectangles, unpublished manuscript, 2005.

\bibitem[Lar]{Lar}
P.~Laroche,
Planar $1$-in-$3$ satisfiability is \NP-complete,
in \emph{Proc. ASMICS Workshop on Tilings, ENS Lyon} (1992).

\bibitem[Lev]{Lev}
L.~Levin,
Universal sorting problems,
\emph{Problems Inf. Transm.}~\textbf{9} (1973), 265--266.

\bibitem[LeP]{LeP}
H.~R.~Lewis and C.~H.~Papadimitriou,
\emph{Elements of the theory of computation},
Upper Saddle River, NJ, 1998.

\bibitem[LP]{LP}
L.~Lov\'{a}sz and M.~D.~Plummer,
\emph{Matching theory} (Corrected reprint of the 1986 original),
AMS, Providence, RI, 2009.

\bibitem[LRS]{LRS}
M.~Luby, D.~Randall and A.~Sinclair,
Markov chain algorithms for planar lattice structures,
\emph{SIAM J. Comput.}~\textbf{31} (2001), 167--192.

\bibitem[MP]{MP}
C.~Moore and I.~Pak,
Ribbon tile invariants from the signed area,
\emph{J. Combin. Theory, Ser.~A}~\textbf{98} (2002), 1--16.

\bibitem[MRR]{MRR}
C.~Moore, I.~Rapaport and E.~Remila,
Tiling groups for Wang tiles,
in \emph{Proc.~13th SODA} (2002), 402--411.

\bibitem[MR]{MR}
C.~Moore and J.~M.~Robson,
Hard tiling problems with simple tiles,
\emph{Discrete Comput. Geom.}~\textbf{26} (2001), 573--590.

\bibitem[MuR]{MuR}
W.~Mulzer and G.~Rote,
Minimum-weight triangulations is \NP-hard,
in \emph{Proc.~22nd SOCG} (2006), 1--10.

\bibitem[NW]{NW}
T.~Neary and D.~Woods,
Four small universal Turing machines,
in \emph{Proc.~5th MCU} (2007), 242--254.

\bibitem[Oll]{Oll}
N.~Ollinger,
Tiling the plane with a fixed number of polyominoes,
in \emph{Proc.~3rd LATA} (2009), 638--647.

\bibitem[Pak]{pak-horizons}
I.~Pak,
Tile invariants: New horizons,
\emph{Theor. Comp. Sci.}~\textbf{303} (2003), 303--331.

\bibitem[PY]{PY-domino}
I.~Pak and J.~Yang,
The complexity of generalized domino tilings,
\href{http://arxiv.org/abs/1305.2154}{\texttt{arXiv:1305.2154}}

\bibitem[Pap]{Pap}
C.~H.~Papadimitriou,
\emph{Computational complexity},
Addison-Wesley, Reading, MA, 1994.

\bibitem[Reid1]{Reid:homotopy}
M.~Reid,
Tile homotopy groups,
\emph{Enseign. Math.}~\textbf{49} (2003), 123--155.

\bibitem[Reid2]{Reid:klarner}
M.~Reid,
Klarner systems and tiling boxes with polyominoes,
\emph{J. Combin. Theory, Ser.~A}~\textbf{111} (2005), 89--105.

\bibitem[R\'{e}m1]{Rem1}
\'{E}.~R{\'e}mila,
Tiling groups: new applications in the triangular lattice,
\emph{Discrete Comput. Geom.}~\textbf{20} (1998), 189--204.

\bibitem[R\'{e}m2]{Rem2}
\'{E}.~R\'{e}mila, Tiling a polygon with two kinds of rectangles,
\emph{Discrete Comput. Geom.}~\textbf{34} (2005), 313--330.

\bibitem[Rob1]{Rob0}
P.~J.~Robinson,
Fault-free rectangles tiled with rectangular polyominoes,
in \emph{Combinatorial Mathematics~IX}, (Brisbane, 1981), 372--377.

\bibitem[Rob2]{Rob}
R.~M.~Robinson,
Undecidability and nonperiodicity for tilings of the plane,
\emph{Invent. Math.}~\textbf{12} (1971), 177--209.

\bibitem[Rog]{Rog}
Y.~Rogozhin,
Small universal Turing machines,
\emph{Theor. Comp. Sci.}~\textbf{168} (1996), 215--240.

\bibitem[She]{She}
S.~Sheffield,
Ribbon tilings and multidimensional height functions,
\emph{Trans. AMS}~\textbf{354} (2002), 4789--4813.

\bibitem[Sta]{Sta}
R.~P.~Stanley,
\emph{Enumerative combinatorics}, Vol.~1,
Cambridge University~Press, Cambridge, 1997.

\bibitem[Thu]{Thu}
W.~Thurston,
Conway's tiling groups,
\emph{Amer. Math. Monthly}~\textbf{97} (1990), 757--773.

\bibitem[Val]{Val}
L.~G.~Valiant,
The complexity of enumeration and reliability problems,
\emph{SIAM J.~Comput.}~\textbf{8} (1979), 410--421.

\bibitem[VV]{VV}
L.~G.~Valiant and V.~V.~Vazirani,
\NP~is as easy as detecting unique solutions,
\emph{Theor. Comp. Sci.}~\textbf{47} (1986), 85--93.

\bibitem[Wang]{Wang}
H.~Wang, Games, logic and computers,
in \emph{Scientific American} (Nov.~1965), 98--106.

\bibitem[Yang]{Yang}
J.~Yang,
Ph.D.~thesis,
available at \ts \url{http://tinyurl.com/d4a89ns}


\end{thebibliography}
\end{document}